\documentclass[reqno]{amsart}
\usepackage{amsmath, amssymb, amsthm, epsfig}
\usepackage{hyperref, latexsym}
\usepackage{url}
\usepackage[mathscr]{euscript}
\usepackage{enumerate}
\usepackage{xcolor}
\usepackage{fullpage} 
\usepackage{setspace}
\usepackage{graphicx} 
\usepackage{mathabx}
\usepackage{mathtools}
\usepackage{easytable} 

\allowdisplaybreaks

\def\today{\ifcase\month\or
  January\or February\or March\or April\or May\or June\or
  July\or August\or September\or October\or November\or December\fi
  \space\number\day, \number\year}

 \newtheorem{theorem}{Theorem}
  
 \newtheorem{lemma}[theorem]{Lemma}
 \newtheorem{proposition}[theorem]{Proposition}
 
 \theoremstyle{definition}

 \theoremstyle{remark}

 \newcommand{\mc}{\mathcal}

 \newcommand{\C}{\mathbb{C}}
 \newcommand{\R}{\mathbb{R}}
 \newcommand{\N}{\mathbb{N}}
 
 \newcommand{\Z}{\mathbb{Z}}

 \newcommand{\dt}{\text{\rm d}t}
  \renewcommand{\d}{\text{\rm d}}

 \newcommand{\dx}{\text{\rm d}x}

\newcommand{\pvector}[1]{
  \begin{pmatrix}
    #1
  \end{pmatrix}} 
\newcommand{\ddirac}[1]{
  \,\boldsymbol{\delta}\!\pvector{#1}\!} 
\renewcommand{\d}{{\rm d}} 
\newcommand{\sph}[1]{\mathbb{S}^{#1}}

\usepackage{esint} 
\usepackage{bbm} 
\usepackage{ulem} 

\begin{document}
\title[Stable sharp Fourier restriction]{Stability of sharp Fourier restriction to spheres}

\author[Carneiro, Negro and Oliveira e Silva]{Emanuel Carneiro, Giuseppe Negro and Diogo Oliveira e Silva}

\address{Emanuel Carneiro, 
ICTP - The Abdus Salam International Centre for Theoretical Physics,
Strada Costiera, 11, I - 34151, Trieste, Italy. ORCID 0000-0001-6229-1139.}
\email{carneiro@ictp.it}

\address{Giuseppe Negro, Departamento de Matemática\\ 
Instituto Superior Técnico\\
Av. Rovisco Pais\\ 
1049-001 Lisboa, Portugal. ORCID 0000-0002-4913-7863.} 
\email{giuseppe.negro@tecnico.ulisboa.pt} 

\address{Diogo Oliveira e Silva, Departamento de Matemática\\ 
Instituto Superior Técnico\\
Av. Rovisco Pais\\ 
1049-001 Lisboa, Portugal. ORCID 0000-0003-4515-4049.} 
\email{diogo.oliveira.e.silva@tecnico.ulisboa.pt} 

\allowdisplaybreaks
\numberwithin{equation}{section}

\maketitle  

\begin{abstract} In dimensions $d \in \{3,4,5,6,7\}$, we prove that the constant functions on the unit sphere $\sph{d-1}\subset \R^d$ maximize the weighted adjoint Fourier restriction inequality
$$   \left| \int_{\R^d} |\widehat{f\sigma}(x)|^4\,\big(1 + g(x)\big)\,\d x\right|^{1/4} \leq {\bf C} \, \|f\|_{L^2(\sph{d-1})}\,,$$
where $\sigma$ is the surface measure on $\sph{d-1}$,
for a suitable class of bounded perturbations $g:\R^d \to \C$. In such cases we also fully classify the complex-valued maximizers of the inequality. In the unperturbed setting  ($g = {\bf 0}$), this was established by Foschi ($d=3$) and by the first and third authors ($d \in \{4,5,6,7\}$) in 2015.
Our methods also yield a new sharp adjoint restriction inequality on $\mathbb S^7\subset \R^8$.
\end{abstract}

\setcounter{tocdepth}{1}
\tableofcontents

\section{Introduction}
Let $\sph{d-1} \subset \R^{d}$ be the $(d-1)$-dimensional unit sphere, $d\geq 2$, equipped with the standard surface measure $\sigma = \sigma_{d-1}$ that verifies $\sigma\big(\sph{d-1}\big)= 2\pi^{d/2} /\, \Gamma(d/2)$. For $f \in L^1(\sph{d-1})$, we define the Fourier transform of the measure $f \sigma$ by
\begin{equation}\label{20210526_14:45}
\widehat{f\sigma}(x)=\int_{\sph{d-1}} e^{ix\cdot\omega}f(\omega)  \,\d\sigma(\omega) \ \ ; \ \ (x \in \R^d).
\end{equation}
{The classic Stein--Tomas inequality states that} 
\begin{equation}\label{20210524_16:43}
\big\|\widehat{f\sigma}\big\|_{L^q(\R^d)} \leq {\bf C} \, \|f\|_{L^2(\sph{d-1})},
\end{equation}
{for all $q \geq 2(d+1)/(d-1)$. This was first proved in 1975 by Tomas \cite{To75} in the range $q > 2(d+1)/(d-1)$, and it was observed by Stein shortly after that the endpoint $q = 2(d+1)/(d-1)$ could also be included; Stein's observation initially went unpublished but one can find a proof in~\cite[pp.~326--328]{SteinBeijing} and a historical discussion in~\cite[Ch.~VIII,~\S 5.15]{St93}. Inequality \eqref{20210524_16:43} was later embedded in a larger framework of Fourier restriction inequalities for quadratic surfaces related to partial differential equations, and presented in the 1977 work of Strichartz \cite{St77}. This appeared shortly after the Babenko--Beckner \cite{Ba, Be75} discovery of the sharp Hausdorff--Young inequality, a landmark in the quest for sharp forms of functional inequalities in harmonic analysis. The general theme of sharp Fourier restriction flourished later, with the pioneering works of Ozawa--Tsutsumi \cite{OT98} on the sharp Strichartz inequality for the Schr\"{o}dinger equation in dimension $d=2$, and of Foschi \cite{Fo} on the sharp Strichartz inequalities for the Schr\"{o}dinger equation in dimensions $d\in\{1,2\}$ and for the wave equation in dimensions $d\in \{2,3\}$.}

\smallskip

{Returning to the sphere,  it is conjectured that constant functions maximize the adjoint Fourier restriction inequality \eqref{20210524_16:43} for all $q \geq 2(d+1)/(d-1)$; see \cite[\S 3.1]{FOS17}. So far, this claim has only been established in a few cases, all in low dimensions: in the Stein--Tomas endpoint case $q=4$ and $d=3$ by Foschi \cite{Fo15}; in the cases $q=4$ and $d \in \{4,5,6,7\}$ by Carneiro and Oliveira e Silva \cite{COS15}; and, more recently, in the cases $q=2k$ with $k\geq 3$ an integer, and $d \in \{3,4,5,6,7\}$, by Oliveira e Silva and Quilodr\'{a}n \cite{OSQu2}. Other works related to the sharp adjoint Fourier restriction to the sphere include \cite{BZKTh, CFOST, COSS2, ChSh1, ChSh2, FLS, GoN, OSQu3, OST, OSTZK, Sh16}. A non-exhaustive list of works in sharp Fourier restriction theory includes \cite{BBCH, Car1, ChQu, Fo, Go2, Go3, HZ, Ku, OT98, Sh09a} for the paraboloid (Schr\"odinger equation), \cite{BJO16, BR, Fo, Ne, Qu13, Ram} for the cone (wave equation), \cite{COSS1, COSSS, Jeavons14, OzR14, Qu1} for the hyperboloid (Klein--Gordon equation), and \cite{BV20, BBI15, BrOSQu, CarOlSo, DMPS, FVV, HS12, OS14, OSQu18, OSQu, Sh09b} in other related settings. We refer the reader to the surveys \cite{FOS17, NOST23} for a more detailed account on the latest developments.}

\smallskip

One can place inequality \eqref{20210524_16:43} within a larger program, via the following weighted setup. Given a bounded function $h:\R^d \to \C$, which functions maximize the weighted adjoint Fourier restriction inequality 
\begin{equation}\label{20210524_17:54*}
    \left| \int_{\R^d} \big|\widehat{f\sigma}(x)\big|^q\,h(x)\,\d x\right|^{1/q} \leq {\bf C} \, \|f\|_{L^2(\sph{d-1})}\,\, ?
    \end{equation}
{We remark that inequality \eqref{20210524_17:54*} might hold in a larger range of $q$ than in \eqref{20210524_16:43} provided that the weight $h$ introduces additional decay. Examples of such situations, with radial weights $h$ and exponent $q=2$, have been considered in \cite{BRV, BMS, BSS}. In the present paper we take a different direction, studying the stability of the (known) maximizers of \eqref{20210524_16:43} in the regime where $h$ is close to ${\bf 1}$.}
    
\smallskip

Our terminology here is the usual one: the value of the optimal constant in the inequality \eqref{20210524_17:54*} is 
\begin{equation}\label{20210709_22:16}
{\bf C} = \displaystyle\sup_{{\bf 0} \neq f \in L^2(\sph{d-1})} \frac{  \left| \,\int_{\R^d} \big|\widehat{f\sigma}(x)\big|^q\,h(x)\,\d x\,\right|^{1/q} }{\|f\|_{L^2(\sph{d-1})}}\,,
\end{equation}
and a {\it maximizer} is a function $f \in L^2(\sph{d-1})\setminus\{{\bf 0}\}$ that realizes the supremum on the right-hand side of \eqref{20210709_22:16}. Throughout the paper, we work with the exponent $q=4$ in dimensions $d \in \{3,4,5,6,7\}$, setting
$$h = {\bf 1} + g,$$
with $g \in L^{\infty}(\R^d)$. We are then interested in the sharp form of the inequality
\begin{equation}\label{20210524_17:54}
    \left| \int_{\R^d} \big|\widehat{f\sigma}(x)\big|^4\,h(x)\,\d x\right|^{1/4} \leq {\bf C} \, \|f\|_{L^2(\sph{d-1})}.
    \end{equation}
Recall that when $g = {\bf 0}$ and $d \in \{3,4,5,6,7\}$, the constant functions maximize \eqref{20210524_17:54}. Heuristically it is expected that, if $g$ is sufficiently small, then the constant functions should come close to realizing equality in \eqref{20210524_17:54} in dimensions $d \in \{3,4,5,6,7\}$. We refine this stability statement and prove that, if the perturbation $g$ is sufficiently {\it regular} and {\it small}, as properly described below, then the constant functions continue to be maximizers of \eqref{20210524_17:54} (and, generically, they are the {\it unique} maximizers). This is the content of Theorems \ref{Thm1}, \ref{Thm1b} and \ref{Thm2} below. We note that, although it may seem more natural to consider non-negative weights $h$ in \eqref{20210524_17:54}, our methods do not require this assumption, and the weight $h$ is free to exhibit sign oscillations; we  present an explicit example in the Remark after the statement of Proposition~\ref{prop:tildeLd}.

\smallskip

Our notation for a multi-index is standard, letting $\alpha = (\alpha_1, \alpha_2, \ldots, \alpha_d)$ with each $\alpha_i \in \Z_{\geq 0}$ ($1 \leq i \leq d$), and writing $|\alpha| := \alpha_1 + \ldots + \alpha_d$ and $\partial^{\alpha}g:= \partial^{\alpha_1}_{x_1} \ldots \partial^{\alpha_d}_{x_d}g$. Throughout the paper we assume the following regularity condition on the perturbation $g:\R^d \to \C$:
\begin{itemize}
\item[(R1)] $g \in L^2(\R^d) \cap L^{\infty}(\R^d)$ and its Fourier transform $\widehat{g}$ is {\it radial} and {\it non-negative} on the closed ball $\overline{B_4}  = \{\xi \in \R^d \ ; \ |\xi|\leq 4\}$.
\end{itemize}
There is a second regularity condition in our study, which is related to the smoothness of $\widehat{g}\big|_{\overline{B_4}}$. Here we consider two cases of interest: 
\begin{itemize}
\item[(R2.A)] (Analytic version)  $\widehat{g}\big|_{\overline{B_4}}: \overline{B_4} \to \R$ admits an analytic continuation, {which we call $\mathcal{G}$}, to an open disk in $\C^d$ containing the closed disk $\overline{\mc{D}_{R}} = \{z \in \C^d\ ; \  |z| \leq R\}$ for some radius $R >4$.
\smallskip
\item[(R2.C)]($C^k$-version) $\widehat{g}\big|_{\overline{B_4}}: \overline{B_4} \to \R$ belongs to $C^k(B_4) \cap C^0(\overline{B_4})$, with bounded partial derivatives of order up to $k$ in $B_4$, where\footnote{Recall that $\lfloor x\rfloor$ denotes the greatest integer that is less than or equal to $x$.} 
$$k = k(d) = \lfloor (d+3)/2 \rfloor.$$
\end{itemize}
\noindent {\sc Remark}: In sympathy with \eqref{20210526_14:45}, our normalization for the Fourier transform in $\R^d$ is 
\begin{equation}\label{eq:FTdef}
\widehat{g}(\xi) = \int_{\R^d} e^{i \xi \cdot x} \,g(x)\,\dx.
\end{equation}
\noindent {\sc Remark}: In this paper we shall always work under conditions $(\rm{R1})$ and $({\rm R2.A})$, or under conditions $(\rm{R1})$ and $({\rm R2.C})$. The portion of the distributional Fourier transform $\widehat{g}$ outside $\overline{B_4}$ has no effect on the integral on the left-hand side of \eqref{20210524_17:54} (see \eqref{eq:Q_h_decomp} below) and we can assume without loss of generality that ${\rm supp}(\widehat{g}) \subset \overline{B_4}$. Hence, by Fourier inversion, we may assume when convenient that $g$ itself is radial, real-valued, smooth, and that $g$ and all of its partial derivatives belong to $L^2(\R^d) \cap L^{\infty}(\R^d)$.

\smallskip

Our main results are the following.
{\begin{theorem}[Sharp weighted adjoint Fourier restriction: analytic version]\label{Thm1}
Let $g:\R^d \to \C$ be a function verifying the regularity conditions $(\rm{R1})$ and $({\rm R2.A})$ above, and set $h = {\bf 1} + g$. Let $R>4$ and $\mathcal{G}:\overline{\mc{D}_{R}} \to \C$ be as in condition $({\rm R2.A})$. Then, for each $d \in \{3,4,5,6,7\}$, there exists a positive constant $\mc{A}_{d,R}$ such that the following holds: if 
\begin{equation}\label{20210525_09:45}
\max_{z \in \overline{ \mc{D}_{R}}} \big| \mathcal{G}(z)\big| < \mc{A}_{d,R}\,,
\end{equation}
then the constant functions are maximizers of the weighted adjoint Fourier restriction inequality \eqref{20210524_17:54}. Our constant $\mc{A}_{d,R}$ is effective, given by \eqref{20210729_13:07}. For instance, when $R=5$, we have
$$\mc{A}_{3,5} = {6.49\ldots} \ ; \ \mc{A}_{4,5} = {90.10\ldots} \ ; \ \mc{A}_{5,5} = {363.29\ldots} \ ; \ \mc{A}_{6,5} = {1092.17\ldots} \ ; \ \mc{A}_{7,5} = {2131.26\ldots}\ .$$
Moreover, the limit $\mc{L}_d := \lim_{R \to \infty} \frac{\mc{A}_{d,R}}{R^2}$ exists and is given by \eqref{20210729_14:34}, corresponding to
\begin{equation}\label{20210729_14:40}
\mc{L}_3 = 1.82\ldots \ ; \ \mc{L}_4 =24.55\ldots  \ ; \ \mc{L}_5=98.59\ldots  \ ; \  \mc{L}_6 = 296.25\ldots \ ; \ \mc{L}_7 = 579.20\ldots\ .
\end{equation}
\end{theorem}}

\smallskip

\noindent {{\sc Remark:} As an example, Theorem \ref{Thm1} can be applied to the situation when $g$ verifies $(\rm{R1})$ and $\widehat{g}(\xi) = a + b|\xi|^2$ for $\xi \in \overline{B_4}$, with $|b| < \mc{L}_d$ (see \S \ref{quadweights} below for further discussion of this particular case). In fact, given \eqref{20210729_14:40}, one can choose $R$ large enough so that \eqref{20210525_09:45} holds, where the analytic continuation is $\mathcal{G}(z) = a + b(z_1^2 + \ldots  + z_d^2)$ with $z = (z_1, z_2, \ldots, z_d) \in \C^d$. We emphasize the fact that $\mathcal{G}$ only needs to be an analytic continuation of $\widehat{g}\big|_{\overline{B_4}}$, and not of $\widehat{g}$ itself. It may be the case that $\mc{G} \neq \widehat{g}$ in $B_R \setminus \overline{B_4}$.

{\begin{theorem}[Sharp weighted adjoint Fourier restriction: $C^k$-version]\label{Thm1b} Let $g:\R^d \to \C$ be a function verifying the regularity conditions $(\rm{R1})$ and $({\rm R2.C})$ above, and set $h = {\bf 1} + g$. Then, for each $d \in \{3,4,5,6,7\}$, there exists a positive constant $\mc{C}_{d}$ such that the following holds: if 
\begin{equation}\label{20210710_19:14}
  {\sup_{\xi \in {B}_{4}} \big| \partial^{\alpha}\widehat{g}(\xi)\big| < \mc{C}_{d}\,,}
 \end{equation}
for any multi-index $\alpha \in\Z_{\geq 0}^d$ of the form $\alpha = (\alpha_1, 0, 0 , \ldots, 0)$, with $0 \leq \alpha_1 \leq k(d)$,
then the constant functions are maximizers of the weighted adjoint Fourier restriction inequality \eqref{20210524_17:54}. Our constant $\mc{C}_{d}$ is effective, given by \eqref{20210715_18:58}, corresponding to
\begin{align*}
{\mc{C}_3 = 0.157\ldots \ \ ; \ \ \mc{C}_4= 0.918\ldots \ \ ; \ \ \mc{C}_5 = 0.908\ldots \ \ ; \ \ \mc{C}_6 = 1.099\ldots\ \ ; \ \ \mc{C}_7 = 0.534\ldots\ .}
\end{align*}
\end{theorem}}

One readily notices that we put some effort into making the main {results} not only qualitative but also quantitative. {In general}, we shall see that the $C^k$-condition (R2.C) corresponds to the minimal regularity required on $\widehat{g}$ in order to achieve our goal. Nevertheless, we decided {to state} the slightly more restrictive analytic version {(Theorem \ref{Thm1})} separately since it is already a fruitful source of examples, with a touch of simplicity in its statement and different insights within the course of its independent proof. {Moreover, there are situations, when both Theorems \ref{Thm1} and \ref{Thm1b} are available, in which the bounds coming from Theorem \ref{Thm1} are strictly superior (for instance,  as in the remark after Theorem \ref{Thm1}, where $\widehat{g}\big|_{\overline{B_4}}$ is a polynomial of low degree in the variable $|\xi|^2$). In particular, it is not the case that Theorem \ref{Thm1} follows from Theorem \ref{Thm1b}.}

\smallskip

{Theorems \ref{Thm1} and \ref{Thm1b}} are complemented by the following classification result.
\begin{theorem}[Full classification of maximizers] \label{Thm2}
Under the hypotheses of Theorem \ref{Thm1} {or Theorem \ref{Thm1b}:}
\begin{itemize}
\item[(i)] If $\widehat{g}= \bf{0}$ on the ball $\overline{B_4}$, then the complex-valued maximizers $f \in L^2(\sph{d-1})$ of \eqref{20210524_17:54} are given by $f(\omega) = c\,e^{iy \cdot  \omega}$, where $y \in \R^d$ and $c \in \C \setminus \{0\}$.
\smallskip
\item[(ii)] If $\widehat{g}\neq \bf{0}$ on the ball $\overline{B_4}$, then the constant functions are the unique complex-valued maximizers of \eqref{20210524_17:54}.
\end{itemize}
\end{theorem}
In the unperturbed setting (i.e. $g =  \bf{0}$), the conclusion of Theorem \ref{Thm2} (i) was established in \cite{COS15, Fo15}, and we just record here for the convenience of the reader that this continues to hold when $\widehat{g}= \bf{0}$ on the ball $\overline{B_4}$, by a simple orthogonality argument. The novelty in the classification above occurs in the broad situation of Theorem \ref{Thm2} (ii), where general complex characters $e^{iy\cdot \omega}$, $y\in\R^d \setminus\{0\}$, {\it do not} maximize \eqref{20210524_17:54}, in contrast to the previous case. This is ultimately due to the modulation/translation symmetry of the extension operator, $\widehat{(e^{iy\cdot} f)\sigma}=\widehat{f \sigma}(\cdot+y)$, which is naturally incompatible with the assumed radiality of {$\widehat{g}\big|_{\overline{B_4}}$}. 

\smallskip

There is a myriad of examples of perturbations $g$ that would fit into our framework. {The most naive one is perhaps the Gaussian $\widehat g(\xi) = c\,e^{- {|\xi|^2}/2}$ (so that, in our normalization, $g(x)=c\,(2\pi)^{-d/2}\,e^{-{|x|^2}/2}$), provided $c>0$ is sufficiently small. In this situation, Theorem \ref{Thm1b} typically provides a better bound than Theorem \ref{Thm1} for the admissible range of the parameter $c$, due to the growth of Gaussians along the imaginary axis. A concrete choice which falls within the scope of Theorem \ref{Thm1b} in every dimension $d \in \{3,4,5,6,7\}$ is $c = \frac{1}{11}$.}

\smallskip

More generally, a particularly simple family for the analytic setting of {Theorem \ref{Thm1}} is given by $\widehat{g}(\xi) = P(|\xi|^2)\int_0^{\infty}  e^{-\lambda  |\xi|^2/2} \d\mu(\lambda)$, where $P$ is an appropriate meromorphic function (e.g. a polynomial) and $\mu$ is a suitable non-negative measure on $(0,\infty)$. For instance, one could take $\widehat{g}(\xi) = c \, e^{- |\xi|^2/2} / (36 + |\xi|^2)$, for some sufficiently small constant $c>0$. Observe that this particular $\widehat{g}$ admits an analytic continuation to the open disk $\mc{D}_6$ but not beyond that. Such examples are prototypical of the analytic case in the following sense: {the fact that $\widehat{g}$ is a radial function on $B_4 \subset \R^d$ which admits an analytic continuation $\mc{G}$ to the disk $\mc{D}_4 \subset \C^d$ is equivalent to the existence of a representation of the form
$$\widehat{g}(\xi_1, \xi_2, \ldots, \xi_d) = G\big((\xi_1^2 + \xi_2^2 +\ldots + \xi_d^2)^{1/2}\big)$$
for $\xi = (\xi_1, \xi_2, \ldots, \xi_d) \in B_{4} \subset \R^d$, where $G:\mc{D}_{4} \subset\C \to \C$ is an even analytic function of one complex variable. In fact, given such conditions on $\widehat{g}$, then $G(z)$ is simply $\mc{G}(z, 0, \ldots, 0)$ for $z \in \mc{D}_{4}$. Conversely, given $G: \mc{D}_{4} \subset \C \to \C$ even and analytic, we can write $G(z)= \sum_{\ell=0}^{\infty} c_{\ell} z^{2\ell}$ (with this series being absolutely convergent for any $z$ with $|z| < 4$) and then $\widehat{g}(\xi_1, \xi_2, \ldots, \xi_d) := \sum_{\ell=0}^{\infty} c_{\ell} (\xi_1^2 + \xi_2^2 + \ldots + \xi_d^2)^{\ell}$ defines a radial function on $B_4$ that admits an analytic continuation to $\mc{D}_4 \subset \C^d$.}

\smallskip 

Let us briefly comment on the motivation behind the regularity conditions. In Section \ref{Sec_9} we discuss how the radiality condition in (R1), the non-negativity condition in (R1), and the smallness condition in \eqref{20210525_09:45} and \eqref{20210710_19:14} (associated to (R2.A) and (R2.C), respectively) are reasonable assumptions, in the sense that if one of them is removed, then it is possible to construct explicit examples of perturbations $g$ for which the constant functions do not maximize \eqref{20210524_17:54}. As far as the proof strategy for Theorems \ref{Thm1} and \ref{Thm1b} is concerned, and how the assumptions play a role, we highlight the following aspects. The radiality condition in (R1) is present in order to preserve the natural radial symmetry of the problem. As in the predecessors \cite{COS15, Fo15}, which treat the case $g = {\bf 0}$, the strategy can be divided into three main steps: 
\begin{itemize}
\item[I.] Symmetrization; 
\item[II.] Magical identity and an application of the Cauchy--Schwarz inequality; 
\item[III.] Spectral analysis of a quadratic form. 
\end{itemize}
Step I, in which we reduce the search of maximizers to non-negative and even functions, is similar to the one in \cite{COS15, Fo15}, using the non-negativity condition in (R1). On the other hand, Steps II and III bring new insights. In \cite{Fo15}, Foschi had the elegant idea of introducing what we call a magical geometric identity to deal with the singularity of the two-fold convolution of the surface measure of the sphere at the origin: if $(\omega_1, \omega_2, \omega_3, \omega_4) \in (\sph{d-1})^4$ are such that $\omega_1 + \omega_2 + \omega_3 + \omega_4 = 0$, then
\begin{equation}\label{20210525_11:49}
|\omega_1 + \omega_2|\, |\omega_3 + \omega_4| + |\omega_1 + \omega_3|\, |\omega_2 + \omega_4| + |\omega_1 + \omega_4|\, |\omega_2 + \omega_3| = 4.
\end{equation}
In the presence of a perturbation $g$, one must find the ``correct" magical identity which needs to be applied, a task that in principle is not obvious. We present a new point of view to generate such magical identities, via the underlying partial differential equation (Helmholtz equation) and opportune applications of integration by parts. This general perspective turns out to be amenable to perturbations, and this ultimately enables our progress in Step II (and, as a by-product, we recover \eqref{20210525_11:49} in the case $g = {\bf 0}$). Finally, in Step III we arrive at the analysis of a suitable quadratic form. Conceptually behind our proof lurks the fact that, in the corresponding step in \cite{COS15, Fo15} for $g = {\bf 0}$, there was ``some room to spare", in the sense that certain Gegenbauer coefficients which appeared in connection to the problem were not only less than or equal to zero (which would suffice for the argument that constants are maximizers) but, in fact, strictly less than zero. In order to properly understand, quantify and  take advantage of such heuristics, we bring in the final regularity assumption {(R2.A) in case of Theorem \ref{Thm1}, and (R2.C) in case of Theorem \ref{Thm1b}}, since, in essence, smoothness of $\widehat{g}$ will ultimately yield the required decay of the corresponding Gegenbauer coefficients.

\smallskip

By H\"{o}lder's inequality, one has $\|f\|_{L^2(\sph{d-1})} \leq \sigma(\mathbb S^{d-1})^{1-\frac2p} \|f\|_{L^p(\sph{d-1})}$ for any $p \geq 2$, with equality if $f$ is constant. Hence, under the assumptions of {Theorem \ref{Thm1} or Theorem \ref{Thm1b}}, we see that the constant functions are also maximizers of the family of inequalities 
\begin{equation*}
\left| \int_{\R^d} \big|\widehat{f\sigma}(x)\big|^4\,h(x)\,\d x\right|^{1/4} \leq {\bf C}\, \sigma(\mathbb S^{d-1})^{1-\frac2p}\, \|f\|_{L^p(\sph{d-1})}\,,
\end{equation*}
for any $p \geq 2$, with the same optimal ${\bf C}$ as in \eqref{20210524_17:54}. A related weighted inequality in the regime $p=4$ and $d=3$, with a simpler setup, has been previously suggested by Christ and Shao in \cite[Remark 16.3]{ChSh1}. 

\smallskip

We conclude with an observation that may be of interest for further research in the theme: the methods developed here yield a proof that constants are maximizers of inequality \eqref{20210524_17:54} for certain non-zero perturbations $g$ in dimensions $d\geq 8$. This is a regime where such a result is not yet known in the original unweighted setting $g={\bf 0}$; see \S \ref{quadweights} for the details, where in particular we prove the following result.

 \smallskip
\begin{theorem}[Sharp inequality on $\mathbb S^7$]\label{thm:NewSharpIneq}
    For every $a> a_\star:=\tfrac{2^{25}\pi^2}{5^2  7^2  11}$, the following sharp inequality holds:
    \begin{equation}\label{eq:corrector_estimate}
        \int_{\mathbb R^8} \lvert \widehat{f\sigma}(x)\rvert^4\, \d x +a\left\lvert \int_{\mathbb S^7}  f(\omega)\, \d \sigma(\omega)\right\rvert^4 \le {\bf C}_a \left( \int_{\mathbb S^7} \lvert f(\omega)\rvert^2\, \d \sigma(\omega)\right)^2,
    \end{equation}
    with optimal constant given by
    \begin{equation}\label{eq:sharp_const_a}
        {\bf C}_a=\int_{\mathbb R^8} \widehat{\sigma}(x)^4\frac{\d x}{\sigma(\mathbb S^7)^2} + a \sigma(\mathbb S^7)^2.
    \end{equation}
    Moreover, equality in~\eqref{eq:corrector_estimate} occurs if and only if $f$ is a constant function on $\mathbb S^7$.
\end{theorem}
\noindent {\sc{Remark}:} We conjecture that~\eqref{eq:corrector_estimate} should hold for every $a\ge 0$, with a larger set of maximizers if $a=0$.

\subsection*{A word on notation} Throughout the text we denote by ${\bf 1}$ (resp. ${\bf 0}$)  the constant function equal to $1$ (resp. $0$), which may be on $\R^d$ or $\sph{d-1}$ depending on the context. The indicator function of a set $X$ is denoted by ${\bf 1}_X$. 
Given a radius $R>0$, we let $B_R = \{x \in \R^d \ ; \ |x| <R\}$ be the open ball centered at the origin in $\R^d$, and $\mc{D}_R = \{z \in \C^d \ ; \ |z| <R\}$ be the open disk (we shall use here the term ``disk" instead of ``ball" just to emphasize the different environment) centered at the origin in $\C^d$. Their respective topological closures are denoted by $\overline{B_R}$ and $\overline{\mc{D}_R}$. We denote by $\widehat{g}\big|_{B_4}$ the restriction of $\widehat{g}$ to the ball $B_4$. We write $A \lesssim B$ if $A \leq C B$ for a certain constant $C>0 $, and we write $A \simeq B$ if $A \lesssim B$ and $B \lesssim A$ (parameters of dependence of such a constant  $C>0$ might appear as a subscript in the inequality sign).
\section{Symmetrization} 
Throughout the paper we keep the notation $h = {\bf 1} + g$. Then 
\begin{equation}\label{eq:hHat}
\widehat{h} = (2\pi)^d\,\boldsymbol{\delta} + \widehat{g}\,,
\end{equation}
where $\boldsymbol{\delta}$ is the $d$-dimensional Dirac delta distribution. For functions $f_i :\sph{d-1} \to \C$ ($1 \leq i \leq 4$), define the quadrilinear form
\begin{equation}\label{20210526_15:39}
 Q_h(f_1,f_2,f_3,f_4):=
\int_{(\sph{d-1})^4} \widehat h(\omega_1+\omega_2-\omega_3-\omega_4)  \prod_{j=1}^4 f_j(\omega_j) \, \d\sigma(\omega_j).
\end{equation}
Further define the quadrilinear forms $Q_{\bf 1}$ and $Q_g$ as in \eqref{20210526_15:39}, with ${\bf 1}$ and $g$ replacing $h$, respectively. Let $f \in L^2(\sph{d-1})$. Plancherel's identity leads us to 
\begin{align}\label{20210717_00:33}
 \int_{\R^d} \big|\widehat{f\sigma}(x)\big|^4\,\d x = (2\pi)^d\,\big\|f\sigma* f\sigma\big\|_{L^2(\R^d)}^2 = Q_{\bf 1}(f, f, \overline{f}, \overline{f})
 \end{align}
(note that this quantity is always non-negative), and
\begin{align}\label{20210717_00:34}
\int_{\R^d} \big|\widehat{f\sigma}(x)\big|^4\,g(x)\,\d x = Q_g(f, f, \overline{f}, \overline{f})
\end{align}
(this quantity, in principle, could be negative). Adding \eqref{20210717_00:33} and \eqref{20210717_00:34} we plainly get
\begin{align}\label{eq:Q_h_decomp}
 \int_{\R^d} \big|\widehat{f\sigma}(x)\big|^4\,h(x)\,\d x = Q_{\bf 1}(f, f, \overline{f}, \overline{f}) + Q_g(f, f, \overline{f}, \overline{f}) = Q_h(f, f, \overline{f}, \overline{f}).
\end{align} 

We now show how to exploit the symmetries of the problem, thus obtaining some estimates which simplify the search for the maximizers. The discussion of the cases of equality will be applied in Section~\ref{sec:classification}.

\subsection{Reduction to non-negative functions} Our first auxiliary result is the following.
\begin{lemma}\label{Lem2_non-negative}
Let $f \in L^2(\sph{d-1})$. We have
\begin{equation}\label{eq:non_neg_symmetry}
\big|Q_h(f, f, \overline{f}, \overline{f})\big| \leq Q_h(|f|, |f|, |f|, |f|). 
\end{equation}
Equality holds if $f$ is non-negative $($in particular, if $f ={\bf 1}$$)$. Furthermore, if there is equality, then necessarily 
\begin{equation}\label{eq:equality_case_non-negative}
    \lVert f\sigma \ast f\sigma \rVert_{L^2(\mathbb R^d)}=\lVert \lvert f\rvert \sigma \ast \lvert f\rvert \sigma \rVert_{L^2(\mathbb R^d)}.
\end{equation}
\end{lemma}
\begin{proof}
Inequality \eqref{eq:non_neg_symmetry} follows immediately from the definition \eqref{20210526_15:39} of $Q_h$ since, by condition (R1) and \eqref{eq:hHat}, we have that the measure $\widehat{h}$ is non-negative on $\overline{B_4}$. Similarly, note that 
\begin{align}\label{20210717_00:50}
Q_{\bf 1}(f, f, \overline{f}, \overline{f}) \leq Q_{\bf 1}(|f|, |f|, |f|, |f|)  \ \ {\rm and} \ \ \big|Q_g(f, f, \overline{f}, \overline{f})\big|  \leq  Q_{g}(|f|, |f|, |f|, |f|).
\end{align}
From \eqref{eq:Q_h_decomp}, the triangle inequality and \eqref{20210717_00:50} we actually have the intermediate inequalities
\begin{align}\label{20210717_00:53}
\begin{split}
\big|Q_h(f, f, \overline{f}, \overline{f})\big| & \leq Q_{\bf 1}(f, f, \overline{f}, \overline{f}) + \big|Q_g(f, f, \overline{f}, \overline{f})\big| \\
& \leq Q_{\bf 1}(|f|, |f|, |f|, |f|) + Q_{g}(|f|, |f|, |f|, |f|) = Q_{h}(|f|, |f|, |f|, |f|).
\end{split}
\end{align}
In order to have equality in \eqref{20210717_00:53}, we must have equality in both inequalities of \eqref{20210717_00:50}, and the first of these is equivalent to \eqref{eq:equality_case_non-negative}.  
\end{proof}
From now on, unless otherwise stated, we will assume without loss of generality that $f$ is a non-negative function. In particular,  $\int_{\R^d} \big|\widehat{f\sigma}(x)\big|^4\,h(x)\,\d x = Q_h(f, f, f, f)$ is also non-negative.

\subsection{Reduction to even functions} Given a function $f:\sph{d-1} \to \R_{\geq 0}$ we define its {\it antipodally symmetric rearrangement} $f_{\sharp}$ by
$$f_{\sharp}(\omega) := \left( \frac{f(\omega)^2 +f(-\omega)^2}{2}\right)^{\frac12}.$$ 
Observe that $\|f_{\sharp}\|_{L^2(\sph{d-1})} = \|f\|_{L^2(\sph{d-1})}$. 
\begin{lemma}\label{Lem3_even}
If $f:\sph{d-1} \to \R_{\geq 0}$ belongs to $L^2(\sph{d-1})$ then
\begin{equation}\label{eq:Q_sharp_symmetrization}
Q_h(f,f,f,f) \leq Q_h(f_{\sharp} , f_{\sharp}, f_{\sharp}, f_{\sharp}).
\end{equation}
There is equality if and only if $f=f_\sharp$ $($in particular, if $f ={\bf 1}$$)$.
\end{lemma}
\begin{proof} Let us abbreviate the notation by writing $\d \sigma (\vec\omega) := \d\sigma(\omega_1) \,\d\sigma(\omega_2)\,\d\sigma(\omega_3)\,\d\sigma(\omega_4)$, with the vector $\vec\omega=(\omega_1,\omega_2,\omega_3,\omega_4)\in(\sph{d-1})^4$. By changing variables and reordering we observe that 
\begin{align*}
& Q_h(f,f,f,f)  = \int_{(\sph{d-1})^4} \widehat{h}(\omega_1+\omega_2-\omega_3-\omega_4)\,\left(\frac{f(\omega_1)f(\omega_3)+f(-\omega_1)f(-\omega_3)}{2}\right)\,f(\omega_2)\,f(\omega_4)\, \d\sigma(\vec\omega)\\
&  \ \ \ \ \leq \int_{(\sph{d-1})^4} \,\widehat{h}(\omega_1+\omega_2-\omega_3-\omega_4)\, \left( \frac{f(\omega_1)^2 +f(-\omega_1)^2}{2}\right)^{\frac12}\!\left( \frac{f(\omega_3)^2 +f(-\omega_3)^2}{2}\right)^{\frac12} \!\!f(\omega_2)\,f(\omega_4) \, \d\sigma(\vec\omega)\\
& \ \ \ \ = Q_h(f_{\sharp}, f, f_{\sharp}, f)\,,
\end{align*}
where we used the Cauchy--Schwarz inequality in its simplest form, ${AB+CD}\leq \sqrt{A^2+C^2}\,\sqrt{B^2+D^2}$. Recall that both the measure ${\bf 1}_{\overline{B_4}}\,\widehat h$ and the function $f$ are non-negative. Repeating the argument with the variables $(\omega_2,\omega_4)$ instead of $(\omega_1,\omega_3)$ finishes the proof of~\eqref{eq:Q_sharp_symmetrization}. 

\smallskip

To establish the case of equality, we note that, by the same argument as in the proof of Lemma~\ref{Lem2_non-negative}, if there is equality in~\eqref{eq:Q_sharp_symmetrization} then $\lVert f\sigma \ast f\sigma\rVert_{L^2(\mathbb R^d)}=\lVert f_\sharp \sigma \ast f_\sharp \sigma \rVert_{L^2(\mathbb R^d)}$. This implies $f=f_\sharp$; see~\cite[Lemma~9]{COS15}. \end{proof}

Hence, on top of being non-negative, we may further assume that $f$ is even (i.e. $f(\omega) = f(-\omega)$ for all $\omega \in \sph{d-1}$) in our search for maximizers. In this case we note that $\widehat{f \sigma}$ is real-valued, and that
\begin{equation}\label{20210526_17:05}
 \int_{\R^d} \big(\widehat{f\sigma}(x)\big)^4\,h(x)\,\d x = Q_h(f,f,f,f)=\int_{(\sph{d-1})^4} \widehat h\left(\sum_{j=1}^4\omega_j\right) \prod_{j=1}^4 f(\omega_j)\, \d\sigma(\omega_j).
\end{equation}

\section{Magical identities via partial differential equations}\label{MIPDE}
Since $\widehat{h} \geq 0$ on $\overline{B_4}$, one could think of applying the Cauchy--Schwarz inequality to the right-hand side of \eqref{20210526_17:05}. This turns out to be an inappropriate move, which leads to an unbounded quadratic form because of the singularity of the two-fold convolution of the surface measure of the sphere at the origin (one has $(\sigma*\sigma)(x) \simeq 1/|x|$ near $x =0$; see \eqref{eq:twofoldconv} below). In order to overcome this obstacle, Foschi \cite{Fo15} had the remarkable idea of introducing a suitable term on the right-hand side of \eqref{20210526_17:05} (with $g = {\bf 0}$) in order to control this singularity. Such a move is only admissible because of the insightful geometric identity \eqref{20210525_11:49}. Our goal in this section is to find a proper replacement in the general weighted situation. We do so by presenting a different perspective on how to look for such magical identities, via the connection with the underlying Helmholtz equation.  

\smallskip 

\subsection{Helmholtz equation and integration by parts}  \label{Helmholtz}The next result lies at the genesis of our magical identity. Recall from the remark after \eqref{eq:FTdef} that we may assume that ${\rm supp}(\widehat{g}) \subset \overline{B_4}$, which implies that $g$ is radial, real-valued, smooth, and that $g$ and all of its derivatives belong to $L^2(\R^d) \cap L^{\infty}(\R^d)$. For simplicity, let us assume this is the case throughout \S\ref{Helmholtz} and \S\ref{MagicIdentity}.

\begin{proposition}\label{prop:magic_id}
{Let $d \geq 3$} and let $f\in L^2(\sph{d-1})$ be a non-negative and even function. Then
\begin{equation}\label{eq:MI}
\int_{\R^d} \big(\widehat{f\sigma}(x)\big)^4 \,h(x)\, \d x 
= 
\frac34\int_{\R^d} \big|\nabla\big( (\widehat{f\sigma})^2\big)(x)\big|^2 \,h(x)\,\d x
-\frac14\int_{\R^d} \big(\widehat{f\sigma}(x)\big)^4 \Delta h(x)\,\d x.
\end{equation}
\end{proposition}
\begin{proof}
Set $u:=\widehat{f\sigma}$ and observe that, by dominated convergence, $u \in C^{\infty}(\R^d)$. 
The function $u$ is a classical solution to the Helmholtz equation $\Delta u (x) + u(x)=0$ for all $x\in \mathbb R^d$. 
Also, by the assumptions on $f$, the function $u$ is real-valued. 

\smallskip

Assume for a moment that $f\in C^\infty(\mathbb S^{d-1})$; this extra hypothesis will be removed at the end of the proof. The Helmholtz equation and integration by parts yield 
\begin{equation}\label{eq:MIstep1}
\int_{\R^d} u^4 h = \int_{\R^d} (-\Delta u) \, u^3h = \int_{\R^d} \nabla u\cdot\nabla (u^3 h).
\end{equation}
Note that there are no boundary terms, since
\begin{equation}\label{20210526_18:08}
\lim_{R\to\infty}\int_{\sph{d-1}_R} \left(\nabla u\cdot \frac{\nu}{R}\right)\,u^3 h\,\d\sigma_{d-1, R}(\nu)= 0,
\end{equation}
where ${\sph{d-1}_R} \subset \R^d$ denotes the sphere of radius $R$ centered at the origin, and $\sigma_{d-1, R}$ is its surface measure. 
Identity \eqref{20210526_18:08} follows from the fact that $f\in C^\infty(\mathbb S^{d-1})$, since a well-known stationary phase argument~\cite[Chapter VIII, \S3, Theorem 1]{St93} yields the decay estimate
\begin{equation}\label{20210526_18:09}
\lvert \nabla u(x)\rvert + |u(x)|\lesssim (1+|x|)^{\frac{1-d}2}
\end{equation}
for every $x\in\R^d$. Estimate \eqref{20210526_18:09} and the fact that $h \in L^{\infty}(\R^d)$ plainly imply \eqref{20210526_18:08}. 

\smallskip

Since $\nabla (u^3 h)=(3u^2\nabla u)h+u^3\nabla h$ and $u^2|\nabla u|^2=\frac14|\nabla(u^2)|^2$, further partial integrations from \eqref{eq:MIstep1} yield
\begin{align}\label{20210526_18:04}
\int_{\R^d} u^4 h 
&=\frac34\int_{\R^d} |\nabla(u^2)|^2\, h-\int_{\R^d} u\,\nabla\cdot(u^3\nabla h)=\frac34\int_{\R^d} |\nabla(u^2)|^2h-\frac14\int_{\R^d} u^4\Delta h,
\end{align}
which is the desired identity \eqref{eq:MI}. The last identity in \eqref{20210526_18:04} amounts to realizing that
\[\nabla\cdot(u^3\nabla h)=\nabla(u^3)\cdot\nabla h + u^3\nabla\cdot(\nabla h)=(3u^2\nabla u)\cdot\nabla h+u^3\Delta h\]
and that
\[\int_{\R^d} u^3\,\nabla u\cdot\nabla h
=\frac14\int_{\R^d}\nabla(u^4)\cdot\nabla h
=-\frac14\int_{\R^d} u^4\Delta h.\]
The boundary terms in the preceding partial integrations vanish, for similar reasons to the ones mentioned in \eqref{20210526_18:08}--\eqref{20210526_18:09}; here we are using that $h \in C^2(\R^d)$ with $h, \nabla h, \Delta h \in L^{\infty}(\R^d)$.

\smallskip

It remains to prove that the smoothness assumption on $f$ can be dropped. For any $M>0$ and any $f_1, f_2$ such that  $\lVert f_1\rVert_{L^2(\mathbb S^{d-1})}+\lVert f_2\rVert_{L^2(\mathbb S^{d-1})}\le M$, {we let $u_j :=  \widehat{f_j\sigma}$ and use the fact that $h\in L^\infty(\mathbb R^d)$, together with the Cauchy--Schwarz inequality and Plancherel's identity, to obtain}
\begin{equation*}
    \begin{split}
        &\left| \int_{\mathbb R^{d}}\big\lvert \nabla u_1^2\big\rvert^2 h -  \int_{\mathbb R^d}\big\lvert \nabla u_2^2\big\rvert^2 h\right|^2 \lesssim_h \int_{\mathbb R^d} \left\lvert \nabla \big(u_1^2-u_2^2\big)\right\rvert^2 \, \int_{\mathbb R^d}\left\lvert\nabla \big(u_1^2+u_2^2\big)\right\rvert^2 \\ 
        &=(2\pi)^{2d} \left(\int_{\mathbb R^d} \lvert y \rvert^2 \left\lvert (f_1\sigma\ast f_1\sigma)(y) -(f_2\sigma\ast f_2\sigma)(y) \right\rvert^2\, \d y\right) \left( \int_{\mathbb R^d}\lvert y \rvert^2 \left\lvert (f_1\sigma\ast f_1\sigma)(y)  +(f_2\sigma\ast f_2\sigma)(y) \right\rvert^2\,\d y\right) \\ 
        &\lesssim_{d} \int_{\mathbb R^d} \big\lvert u_1^2 - u_2^2\big\rvert^2 \int_{\mathbb R^d} \big\lvert u_1^2 + u_2^2\big\rvert^2 \lesssim_{d, M}\lVert f_1-f_2\rVert_{L^2(\mathbb S^{d-1})}^2. 
    \end{split}
\end{equation*}
In the last line, we have used the fact that both $f_1\sigma\ast f_1\sigma$ and $f_2\sigma\ast f_2 \sigma$ {are supported on $\overline{B_2}$}, as well as the Stein--Tomas estimate; {in fact, note that 
\begin{align*}
\int_{\mathbb R^d} \big\lvert u_1^2 - u_2^2\big\rvert^2 & = \int_{\mathbb R^d} \big\lvert u_1 + u_2\big\rvert^2 \ \big\lvert u_1 - u_2\big\rvert^2 \leq \left(\int_{\mathbb R^d} \big\lvert u_1 + u_2\big\rvert^4\right)^{1/2}\left(\int_{\mathbb R^d} \big\lvert u_1 - u_2\big\rvert^4\right)^{1/2}\\
& \leq \Big( \|u_1\|_{L^4(\R^d)} +  \|u_2\|_{L^4(\R^d)}\Big)^2 \left(\int_{\mathbb R^d} \big\lvert u_1 - u_2\big\rvert^4\right)^{1/2}\\
& \lesssim_{d} \Big( \|f_1\|_{L^1(\mathbb{S}^{d-1})} +  \|f_2\|_{L^2(\mathbb{S}^{d-1})}\Big)^2\ \lVert f_1-f_2\rVert_{L^2(\mathbb S^{d-1})}^2\\
& \lesssim_{d,M}\lVert f_1-f_2\rVert_{L^2(\mathbb S^{d-1})}^2\,,
\end{align*}
where we used the Cauchy–Schwarz inequality, triangle inequality, and finally the Stein--Tomas estimate, as claimed.} This proves that the first term on the right-hand side of~\eqref{eq:MI} is a continuous function of $f\in L^2(\mathbb S^{d-1})$. Since $\Delta h$ is also in $L^\infty(\mathbb R^d)$, the same argument proves that all the terms in~\eqref{eq:MI} are continuous in $L^2(\mathbb S^{d-1})$, which concludes the proof by density. 
\end{proof}

\subsection{Magical identity} \label{MagicIdentity}Taking the gradient (in the variable $x$) in \eqref{20210526_14:45} yields
\begin{equation*}
\nabla\big( (\widehat{f\sigma})^2\big)(x) = i\int_{(\sph{d-1})^2} e^{ix\cdot (\omega_1+\omega_2)} (\omega_1+\omega_2)    f(\omega_1)f(\omega_2) \,\d\sigma(\omega_1) \,\d\sigma(\omega_2).
\end{equation*}
Hence
\begin{align}\label{20210527_12:23}
\begin{split}
\big|\nabla\big( (\widehat{f\sigma})^2\big)(x)\big|^2
&=\int_{(\sph{d-1})^4} e^{ix\cdot(\omega_1+\omega_2-\omega_3-\omega_4)}(\omega_1+\omega_2)\cdot (\omega_3+\omega_4)\,\prod_{j=1}^4 f(\omega_j)  \,\d\sigma(\omega_j)\\
&=-\int_{(\sph{d-1})^4}  e^{ix\cdot\left(\sum_{j=1}^4\omega_j\right)}(\omega_1+\omega_2)\cdot (\omega_3+\omega_4) \prod_{j=1}^4 f(\omega_j)\,\d\sigma(\omega_j)\,,
\end{split}
\end{align}
where we have used the fact that $f$ is non-negative and even. It follows that ($\frac43$ times) the first term on the right-hand side of identity \eqref{eq:MI} is given by
\begin{align}\label{20210527_12:41}
\int_{\R^d} \big|\nabla\big( (\widehat{f\sigma})^2\big)(x)\big|^2 \,h(x)\,\d x = -\int_{(\sph{d-1})^4} \widehat{h}\left(\sum_{j=1}^4\omega_j\right)(\omega_1+\omega_2)\cdot(\omega_3+\omega_4) \prod_{j=1}^4 f(\omega_j)\,\d\sigma(\omega_j).
\end{align}

\smallskip

Similarly, for the second term on the right-hand side of \eqref{eq:MI}, we have
\begin{align}\label{20210527_12:42}
\begin{split}
\int_{\R^d} \big(\widehat{f\sigma}(x)\big)^4 \Delta h(x)\,\d x
&=\int_{(\sph{d-1})^4}\widehat{\Delta h}\left(\sum_{j=1}^4\omega_j\right)\prod_{j=1}^4f(\omega_j)\,\d\sigma(\omega_j)\\
&=-\int_{(\sph{d-1})^4}\widehat{ h}\left(\sum_{j=1}^4\omega_j\right)\left|\sum_{j=1}^4\omega_j\right|^2\prod_{j=1}^4f(\omega_j)\,\d\sigma(\omega_j).
\end{split}
\end{align}

\smallskip

At this point, we note that 
\begin{equation}\label{eq:NonNeg}
-3(\omega_1+\omega_2)\cdot(\omega_3+\omega_4)+\left|\sum_{j=1}^4\omega_j\right|^2
=
|\omega_1+\omega_2|^2+|\omega_3+\omega_4|^2-(\omega_1+\omega_2)\cdot(\omega_3+\omega_4)\geq 0\,,
\end{equation}
in light of the Cauchy--Schwarz and the AM--GM inequalities.
In fact, the left-hand side of \eqref{eq:NonNeg} is zero if and only if $\omega_1 + \omega_2 = \omega_3 + \omega_4 = 0$. Plugging \eqref{20210527_12:41}, \eqref{20210527_12:42} and \eqref{eq:NonNeg} into \eqref{eq:MI} we arrive at
\begin{align}\label{20210527_12:46}
\begin{split}
\int_{\R^d} & \big(\widehat{f\sigma}(x)\big)^4 \,h(x)\, \d x \\
& = \frac14\int_{(\sph{d-1})^4} \widehat h \left(\sum_{j=1}^4\omega_j\right) \left(|\omega_1+\omega_2|^2+|\omega_3+\omega_4|^2-(\omega_1+\omega_2)\cdot(\omega_3+\omega_4)\right) \prod_{j=1}^4 f(\omega_j)\, \d\sigma(\omega_j).
\end{split}
\end{align}
This is our magical identity. Note that when $g = {\bf 0}$ we have 
\begin{equation*}
    \widehat{h}\left(\sum_{j=1}^4 \omega_j\right)\prod_{j=1}^4 \,\d\sigma(\omega_j) = (2\pi)^d\,\boldsymbol{\delta}\left(\sum_{j=1}^4 \omega_j\right)\prod_{j=1}^4 \d\sigma(\omega_j),
\end{equation*}
which is a measure supported in the submanifold  of $(\mathbb S^{d-1})^4$ defined by the equation $\sum_{j=1}^4 \omega_j=0$. With respect to this measure, the newly introduced term \eqref{eq:NonNeg} is almost everywhere equal to a multiple of $|\omega_1+\omega_2|^2 = |\omega_3+\omega_4|^2 = |\omega_1+\omega_2|\, |\omega_3+\omega_4|$, and we thus recover Foschi's identity~\cite[Eq.~(9)]{Fo15}.

\subsection{Cauchy--Schwarz} \label{sec:cauchy_schwarz} Recall that we are assuming $f$ to be non-negative and even. In light of \eqref{eq:NonNeg} and the fact that ${\bf 1}_{\overline{B_4}}\, \widehat h \geq 0$, we are now in position to move on by applying the Cauchy--Schwarz inequality on the right-hand side of \eqref{20210527_12:46}.  This leads to 
\begin{align}\label{20210527_14:23}
\int_{\R^d} & \big(\widehat{f\sigma}(x)\big)^4 \,h(x)\, \d x \leq  \frac14\int_{(\sph{d-1})^2} \,f(\omega_1)^2f(\omega_2)^2 \,\widetilde K_h(\omega_1,\omega_2) \, \d\sigma(\omega_1)\, \d\sigma(\omega_2),
\end{align}
where 
\begin{equation}\label{20210528_11:33}
\widetilde K_h(\omega_1,\omega_2):=\int_{(\sph{d-1})^2}\widehat h \left(\sum_{j=1}^4\omega_j\right)\left(|\omega_1+\omega_2|^2+|\omega_3+\omega_4|^2-(\omega_1+\omega_2)\cdot(\omega_3+\omega_4)\right) \d\sigma(\omega_3)\,\d\sigma(\omega_4).
\end{equation}
Since $\widehat h$ is radial on $\overline{B_4}$, $\widetilde K_h(\omega_1,\omega_2)=\widetilde K_h(\rho\omega_1,\rho\omega_2)$ for every rotation $\rho\in \textup{SO}(d)$ and, therefore, $\widetilde K_h$ depends only on the inner product $\omega_1\cdot\omega_2$.
Thus we define 
\begin{equation}\label{20210528_11:34}
K_h(\omega_1\cdot\omega_2):=\widetilde K_h(\omega_1,\omega_2).
\end{equation} 
Further define the functions $K_{\bf 1}$ and $K_g$ as in \eqref{20210528_11:33}--\eqref{20210528_11:34}, with ${\bf 1}$ and $g$ replacing $h$, respectively.

\smallskip

\noindent{\sc Remark}: Assuming $f$ non-negative and even, equality happens in \eqref{20210527_14:23} if and only if $f$ is constant. To see this, just split $h=\boldsymbol{1} + g$ and argue like in the proof of Lemma~\ref{Lem3_even}, using that the cases of equality for the analogous of~\eqref{20210527_14:23} with $h=\boldsymbol{1}$ have already been completely characterized in~\cite[Lemma 11]{COS15}, and are only the constant functions.

\section{Bilinear analysis} The task ahead of us now consists of analyzing the right-hand side of \eqref{20210527_14:23}. 

\subsection{A quadratic form} \label{sec:quadratic} Consider the quadratic form 
\begin{equation}\label{20210527_14:54}
H_{d,h}(\varphi):=\int_{(\sph{d-1})^2} {\varphi}(\omega_1)\,\overline{\varphi(\omega_2)} \,K_h(\omega_1\cdot \omega_2)\,\d\sigma(\omega_1)\,\d\sigma(\omega_2).
\end{equation}
This defines a real-valued and continuous functional on $L^1(\sph{d-1})$.
Indeed, $K_h=K_{{\bf 1}}+K_g$, and so $H_{d,h}=H_{d,{\bf 1}}+H_{d,g}$.
From \cite[Lemma 5]{COS15} it follows that (recall that $\widehat{{\bf 1}} = (2\pi)^d\,\boldsymbol{\delta}$ in our setup)
\begin{equation}\label{eq:kappad}
K_{\bf 1}(t)= \left[(2\pi)^d\cdot 3\cdot 2^{2-\frac{d}{2}}\sigma_{d-2}\big(\sph{d-2}\big)\right]\,(1+t)^{\frac12}(1-t)^{\frac{d-3}2},
\end{equation}
where $t:=\omega_1\cdot\omega_2$. The continuity of $H_{d,{\bf 1}}$ on $L^1(\sph{d-1})$, as noted in \cite[Eq.\@ (5.19)]{COS15}, is a simple consequence of the fact that $K_{\bf 1} \in L^{\infty}([-1,1])$ since one can prove directly from \eqref{20210527_14:54} that 
\begin{equation*}
\big|H_{d,{\bf 1}}(\varphi_1) - H_{d,{\bf 1}}(\varphi_2)\big| \leq \|K_{\bf 1}\|_{L^{\infty}([-1,1])}\, \left( \|\varphi_1\|_{L^1(\sph{d-1})} + \|\varphi_2\|_{L^1(\sph{d-1})}\right) \, \|\varphi_1 - \varphi_2\|_{L^1(\sph{d-1})}.
\end{equation*}
The continuity of $H_{d,g}$ follows similarly since $\widehat g$ is bounded on $\overline{B_4}$, whence $K_g\in L^\infty([-1,1])$ and an analogous argument applies.  

\smallskip

The following proposition is the final piece in our puzzle.

\begin{proposition}\label{Prop5}
Let $d\in\{3,4,5,6,7\}$. Let $\varphi\in L^1(\sph{d-1})$ be an even function and write
$$\mu=\frac{1}{\sigma\big(\sph{d-1}\big)}\int_{\sph{d-1}}\varphi(\omega)\,\d\sigma(\omega)$$ 
for the average of $\varphi$ over $\sph{d-1}$. {Under the hypotheses of Theorem \ref{Thm1} $($resp. Theorem \ref{Thm1b}$\,)$} there exists a positive constant $\mc{A}_{d,R}$ $($resp. $\mc{C}_d$$)$ such that, if \eqref{20210525_09:45} $($resp. \eqref{20210710_19:14}$)$ holds, then
\begin{equation*}
H_{d,h}(\varphi)\leq H_{d,h}(\mu{\bf 1})=|\mu|^2 H_{d,h}({\bf 1})\,,
\end{equation*}
with equality if and only if $\varphi$ is a constant function.
\end{proposition}

\begin{proof}[{Proof of part of Theorems \ref{Thm1} and \ref{Thm1b}}:] Assuming the validity of Proposition \ref{Prop5}, we apply it with $\varphi = f^2$ (in which $f$ is non-negative and even), coming from  \eqref{20210527_14:23}, to obtain
\begin{equation}\label{20210528_08:09}
\int_{\R^d}  \big|\widehat{f\sigma}(x)\big|^4 \,h(x)\, \d x \leq \frac{H_{d,h}({\bf 1})}{4 \, \sigma\big(\sph{d-1}\big)^2}\,  \|f\|_{L^2(\sph{d-1})}^4.
\end{equation}
{Note that the left-hand side is non-negative; recall~\eqref{20210526_17:05}}. The fact that \eqref{20210528_08:09} holds for all $f \in L^2(\sph{d-1})$ (with the absolute value of the integral on the left-hand side) follows from Lemmas \ref{Lem2_non-negative} and \ref{Lem3_even}. Note that equality holds in \eqref{20210528_08:09} if $f = {\bf 1}$. This establishes the claim of {Theorems \ref{Thm1} and \ref{Thm1b}} that constant functions maximize the weighted adjoint restriction inequality \eqref{20210524_17:54}. 
\end{proof}
We shall first discuss the quantitative part of Theorems \ref{Thm1} and \ref{Thm1b}, and then return to the full characterization of maximizers in Section~\ref{sec:classification} below.

\subsection{Proof of Proposition \ref{Prop5}} In order to prove Proposition \ref{Prop5}, we may work without loss of generality with $\varphi\in L^2(\sph{d-1})$. The general case, including the characterization of the cases of equality, follows by a density argument as outlined in our precursor \cite[Proof of Lemma 12]{COS15}, using the continuity of $H_{d,h}$ in $L^1(\sph{d-1})$.

\subsubsection{Funk--Hecke formula and Gegenbauer polynomials} If $\varphi\in L^2(\sph{d-1})$ we write
\begin{equation}\label{eq:SH}
\varphi=\sum_{n=0}^\infty Y_n\,,
\end{equation}
where $Y_n$ is a spherical harmonic of degree $n$. Since $\varphi$ is an even function, in the representation \eqref{eq:SH} we must have $Y_{2\ell +1}=0$ for all $\ell \in \Z_{\geq 0}$. Note also that $Y_0=\mu {\bf 1}$. The partial sums $\sum_{n=0}^N Y_n$ converge to $\varphi$ in $L^2(\sph{d-1})$, as $N\to\infty$, and hence also in $L^1(\sph{d-1})$. Therefore, from \eqref{20210527_14:54} and \eqref{eq:SH}, we are led to
\begin{align}\label{20210527_10:50}
H_{d,h}(\varphi) = \lim_{N \to \infty} \sum_{m,n=0}^{N} \int_{\sph{d-1}} \overline{Y_{m}(\omega_2)} \left( \int_{\sph{d-1}} Y_{n}(\omega_1) \,K_h(\omega_1\cdot \omega_2)\,\d\sigma(\omega_1)\right)\,\d\sigma(\omega_2). 
\end{align}
The tool to evaluate the latter inner integral is the Funk--Hecke formula \cite[Theorem~1.2.9]{DaiXu}:
\begin{equation}\label{20210527_10:51}
\int_{\sph{d-1}} Y_{n}(\omega_1) \,K_h(\omega_1\cdot \omega_2)\,\d\sigma(\omega_1) = \lambda_{d,h}(n)\,Y_{n}(\omega_2)\,,
\end{equation}
with the constant $\lambda_{d,h}(n)$ given by
\begin{equation}\label{20210528_11:59}
\lambda_{d,h}(n) =\sigma_{d-2}\big(\sph{d-2}\big)\int_{-1}^1  \frac{C_{n}^{\frac{d-2}{2}}(t)}{C_{n}^{\frac{d-2}{2}}(1)} \, K_h(t) \,(1-t^2)^{\frac{d-3}2} \,\d t.
\end{equation}
Here, $C_{n}^{\frac{d-2}{2}}$ denotes the {\it Gegenbauer polynomial} ({\it or ultraspherical polynomial}) of degree $n$ and order $\frac{d-2}{2}$. In general, for $\alpha >0$, the Gegenbauer polynomials $t\mapsto C_n^\alpha(t)$ are defined via the generating function 
\begin{equation}\label{eq:GenFuncGP}
(1-2rt+r^2)^{-\alpha}=\sum_{n=0}^\infty C_n^\alpha(t) \, r^n.
\end{equation}
Note that, if $t\in[-1,1]$, the left-hand side of \eqref{eq:GenFuncGP} defines an analytic function of $r$ (for small $r$) and the right-hand side of \eqref{eq:GenFuncGP} is the corresponding power series expansion. We further remark that $C^{\alpha}_{n}(t)$ has degree $n$, and that the Gegenbauer polynomials $\big\{C^{\alpha}_{n}(t)\big\}_{n=0}^{\infty}$ are orthogonal in the interval $[-1,1]$ with respect to the measure $(1- t^2)^{\alpha - \frac12}\,\dt$. Differentiating \eqref{eq:GenFuncGP} with respect to the variable $r$ and comparing coefficients, we obtain the following three-term recursion relation, valid for any $n\geq 1$:
\begin{equation*}
2t(n+\alpha)C_n^\alpha(t)=(n+1)C_{n+1}^\alpha(t)+(n+2\alpha-1)C_{n-1}^\alpha(t)\,,
\end{equation*}
which coincides with \cite[Eq.\@ (2.1)]{Wa16}. Since, additionally, $C_0^\alpha(t)=1$ and $C_1^\alpha(t)=2\alpha t$, our normalization agrees with that from \cite{Wa16}, which is going to be used later in some of our effective estimates. In this normalization,
\begin{equation}\label{eq:ValueAt1}
C_{n}^\alpha(1)=\frac{\Gamma(n+2\alpha)}{n!\, \Gamma(2\alpha)} \ \ \ {\rm and} \ \ \ \int_{-1}^1 C_{n}^\alpha(t)^2\,(1- t^2)^{\alpha - \frac12}\,\dt =\frac{2^{1-2\alpha}\,\pi}{\Gamma(\alpha)^2}\frac{\Gamma(n + 2\alpha)}{n!\, (n + \alpha)}=: (\frak{h}_{n}^{\alpha})^2.
\end{equation}
We further note that $C_{n}^\alpha(-t)=(-1)^n C_{n}^\alpha(t)$ and that $\max_{t \in [-1,1]} \big|C_{n}^\alpha(t)\big| = C_{n}^\alpha(1)$; see \cite[Theorem  7.33.1]{Sz}.

\smallskip

Returning to our discussion, since spherical harmonics of different degrees are pairwise orthogonal, we plainly get from \eqref{20210527_10:50},  \eqref{20210527_10:51}, and the fact that $Y_{n} =0$ if $n$ is odd, that 
\begin{equation*}
H_{d,h}(\varphi)=\sum_{\ell=0}^\infty \lambda_{d,h}(2\ell)\, \|Y_{2\ell}\|_{L^2(\sph{d-1})}^2.
\end{equation*}
The crux of the matter lies in the following result.
\begin{lemma}[Signed coefficients]\label{Lem6_crux}
Let $d\in\{3,4,5,6,7\}$. {Under the hypotheses of Theorem \ref{Thm1} $($resp. Theorem \ref{Thm1b}$\,)$} there exists a positive constant $\mc{A}_{d,R}$ $($resp. $\mc{C}_d$$)$ such that, if \eqref{20210525_09:45} $($resp. \eqref{20210710_19:14}$)$ holds, then $\lambda_{d,h}(0)>0$ and $\lambda_{d,h}(2\ell)<0$ for every $\ell \geq 1$.
\end{lemma}
Assuming the validity of Lemma \ref{Lem6_crux}, the proof of Proposition \ref{Prop5} follows at once since
\begin{equation*}
H_{d,h}(\varphi)=\sum_{\ell=0}^\infty \lambda_{d,h}(2\ell)\, \|Y_{2\ell}\|_{L^2(\sph{d-1})}^2 \leq  \lambda_{d,h}(0)\|Y_{0}\|_{L^2(\sph{d-1})}^2 = H_{d,h}(\mu{\bf 1})=|\mu|^2 H_{d,h}({\bf 1}),
\end{equation*}
with equality if and only if $Y_{2\ell} =0$ for all $\ell \geq 1$, which means that $\varphi = Y_0$ is a constant function.

\smallskip

We address the proof of the key Lemma \ref{Lem6_crux} in the next three sections.

\section{The spectral gap}\label{sec:spectralgap}
In this section, we briefly discuss the common strategy for the proof {of Lemma \ref{Lem6_crux}}, both in the analytic and $C^k$-versions, and quantify the available gap. Throughout the rest of the paper we let $\nu := \tfrac{d-2}{2}$. For $n \in \Z_{\geq 0}$, define the coefficients $\lambda_{d,{\bf 1}}(n)$ and $\lambda_{d,g}(n)$ as in \eqref{20210528_11:59}, with $K_{\bf 1}$ and $K_g$ replacing $K_h$, respectively. From condition (R1), observe that $K_h \geq K_{\bf 1} > 0$ in $(-1,1)$. Since $C_0^{\nu}(t) = 1$ we plainly get that $\lambda_{d,h}(0) \geq \lambda_{d,{\bf 1}}(0) >0$.

\subsection{The strategy} For $\ell\geq 1$ we proceed as follows. First we write
\begin{equation}\label{20210710_23:01}
\lambda_{d,h}(2\ell)=\lambda_{d,{\bf 1}}(2\ell)+\lambda_{d,g}(2\ell).
\end{equation}
The following observation from the proof of \cite[Lemma 13]{COS15} is a key ingredient in our argument: for each $d\in\{3,4,5,6,7\}$ there exists a constant $c_d >0$ such that, for every $\ell \geq 1$,
\begin{equation}\label{eq:Bounds1}
 \lambda_{d,{\bf 1}}(2\ell)\leq -c_d\,\ell^{-d}<0\,;
\end{equation}
see Lemma \ref{Lem_previous_asymp} below for a precise quantitative statement.\footnote{If $d\geq 8$, then we start to observe that $\lambda_{d,{\bf 1}}(2)>0$ and this step of the proof breaks down.} In order to argue that \eqref{20210710_23:01} is negative for all $\ell \geq 1$, in light of \eqref{eq:Bounds1} it suffices to show that, for all $\ell \geq 1$, we have
\begin{equation}\label{eq:Compareg1}
|\lambda_{d,g}(2\ell)|<c_d\,\ell^{-d}.
\end{equation}

If we consider the Gegenbauer expansion of $K_g$, namely,
\begin{equation}\label{20210528_14:54}
K_g(t)=\sum_{n=0}^\infty a_{n}^{\nu} \, C_{n}^\nu(t) \ \ \ (t \in [-1,1]),
\end{equation}
we find directly from \eqref{20210528_11:59} and \eqref{eq:ValueAt1} that
\begin{equation}\label{20210528_14:55}
\lambda_{d,g}(n) = \frac{2 \,\pi^{\nu +1}}{(n + \nu) \, \Gamma(\nu)} \ a_{n}^{\nu}.
\end{equation} 
Here we used the fact that $\sigma_{d-2}\big(\sph{d-2}\big)= 2\,\pi^{\nu + \frac12} /\, \Gamma(\nu + \tfrac12)$ together with the duplicating formula for the Gamma function, $\Gamma(\nu) \Gamma(\nu + \tfrac12) = 2^{1 - 2\nu} \,\pi^{\frac12}\, \Gamma(2\nu)$. Looking back at \eqref{eq:Compareg1}--\eqref{20210528_14:55}, we see that we need good estimates for the decay of the Gegenbauer coefficients $a_{2\ell}^{\nu}$ in terms of the function $K_g$, and this is ultimately where the smoothness of $\widehat{g}\big|_{\overline{B_4}}$ will play a role.

\subsection{Quantifying the gap}\label{sec:Gap}
We now provide an effective form of the gap inequality \eqref{eq:Bounds1}. Most of the work towards this goal was essentially accomplished in \cite{COS15, Fo15}, and here we just revisit it in a format that is appropriate for our purposes. For convenience, let us recall \eqref{eq:kappad} and \eqref{20210528_11:59}:
\begin{align*}
K_{\bf 1}(t)& = \left[(2\pi)^d\cdot 3\cdot 2^{2-\frac{d}{2}}\sigma_{d-2}\big(\sph{d-2}\big)\right]\,(1+t)^{\frac12}(1-t)^{\frac{d-3}2}\, \ \ \ (t \in [-1,1]);\\
\lambda_{d,{\bf 1}}(n) & =\sigma_{d-2}\big(\sph{d-2}\big)\int_{-1}^1  \frac{C_{n}^{\frac{d-2}{2}}(t)}{C_{n}^{\frac{d-2}{2}}(1)} \, K_{\bf 1}(t) \,(1-t^2)^{\frac{d-3}2} \,\d t\ \ \ \ (n \in \Z_{\geq 0}).
\end{align*}

\begin{lemma}[cf. \cite{COS15, Fo15}]\label{Lem_previous_asymp} For $d \in \{3,4,5,6,7\}$ let $\kappa_d := (2\pi)^d\cdot 3\cdot 2^{3-d}\left[\sigma_{d-2}\big(\sph{d-2}\big)\right]^2$. Then, for each $\ell \geq 1$,
\begin{align}
\lambda_{3,{\bf 1}}(2\ell) & =\kappa_3 \left( \frac{-8}{(4\ell-1)(4\ell+1)(4\ell+3)} \right)\leq -\left(\frac{2^8 \pi^5}{35}\right) \ell^{-3}\,; \label{20210601_09:57}\\ 
\lambda_{4,{\bf 1}}(2\ell) & =\kappa_4 \left( \frac{-8}{(2\ell-1)\,(2\ell+1)^2\,(2\ell+3) }\right) \leq -\left(\frac{2^{10} \pi^6}{15}\right) \ell^{-4}\,; \label{20210601_09:58}\\
\lambda_{5,{\bf 1}}(2\ell) & =\kappa_5\left(\!\frac{-1536(2\ell+1)(2\ell+2)(4\ell^2 + 6\ell - 3)}{\binom{2\ell+2}{2}(4\ell-3)(4\ell-1)(4\ell+1)(4\ell+3)(4\ell+5)(4\ell+7)(4\ell+9)}\right)\!\leq -\left(\frac{2^{15} \pi^9}{2145}\right)\ell^{-5}\,;\label{20210601_09:59}\\
\lambda_{6,{\bf 1}}(2\ell) & =\kappa_6\left(\frac{-32(2\ell+2)}{\binom{2\ell+3}{3}(2\ell-1)(2\ell+1)(2\ell+3)(2\ell+5)}\right)\leq -\left(\frac{2^{15} \pi^{10}}{1575}\right)\ell^{-6}\,;\label{20210601_09:60}\\
\lambda_{7,{\bf 1}}(2\ell) &=\kappa_7\left(\frac{-163840 (2\ell+1)(2\ell+2)(2\ell+3)(2\ell+4)(4\ell^2 + 10\ell - 15)(4\ell^2 + 10\ell - 3)}{\binom{2\ell+4}{4} (4\ell\!-\!5)(4\ell\!-\!3)(4\ell\!-\!1)(4\ell\!+\!1)(4\ell\!+\!3)(4\ell\!+\!5)(4\ell\!+\!7)(4\ell\!+\!9)(4\ell\!+\!11)(4\ell\!+\!13)(4\ell\!+\!15)}\right) \nonumber \\
& \qquad   \leq -\left(\frac{2^{21} \pi^{13}}{1322685 }\right)\ell^{-7}.\label{20210601_09:61}
\end{align}
\end{lemma}
\begin{proof} The identity in \eqref{20210601_09:57} follows from \cite[Proof of Lemma 5.4]{Fo15} (in the notation of that proof, one has $C_{2 \ell}^{1/2} = P_{2\ell}$, which is an even function). The identities in \eqref{20210601_09:58}--\eqref{20210601_09:61} follow from \cite[Proof of Lemma 13, Steps 2 to 5]{COS15} (in the notation of that proof, there is a quantity $\Lambda_{2\ell}(\phi_d)$ which is computed, satisfying $\lambda_{d,{\bf 1}}(2\ell) = \big(\kappa_d / \sigma_{d-2}\big(\sph{d-2}\big)\big)\Lambda_{2\ell}(\phi_d)$, with $\lambda_{d,{\bf 1}}(2\ell)$ and $\kappa_d$ as defined above; recall that $C_{2\ell}^{\frac{d-2}{2}}$ is even). 

\smallskip

The upper bounds in \eqref{20210601_09:57}--\eqref{20210601_09:61} work as follows. For \eqref{20210601_09:57}, one multiplies the left hand-side by the appropriate power of $\ell$ (in this case, $\ell^3$) and observe that 
$$\ell \mapsto  \frac{-8\ell^3}{(4\ell-1)(4\ell+1)(4\ell+3)}$$
defines a decreasing function of $\ell \geq 1$. This is a routine verification (e.g.\@ with basic computer aid). The upper bound then comes from evaluating it at $\ell =1$. The other cases follow the same reasoning. 
\end{proof}

\section{Proof of Lemma \ref{Lem6_crux}: analytic version}
{In this section we work under the hypotheses of Theorem \ref{Thm1}; in particular, (R2.A) holds. Recall $\nu := \tfrac{d-2}{2}$}.

\subsection{Bounds for the Gegenbauer coefficients (analytic version)} As previously observed, we need decay estimates for the Gegenbauer coefficients $a_{n}^{\nu}$ in terms of the function $K_g$ in \eqref{20210528_14:54}. The analogous situation for Fourier series is very classical, via the paradigm that regularity of the function implies decay of the Fourier coefficients. Here we face a similar situation, where the orthogonal basis is the one of Gegenbauer polynomials, and we want to deploy the same philosophy that regularity on one side implies decay on the other side. 

\smallskip

If our function, initially defined on the interval $[-1,1]$, admits an analytic continuation, then we will be able to invoke careful quantitative estimates from the recent work of Wang \cite{Wa16}. In order to state the relevant result, given $\rho>1$, define the so-called {\it Bernstein ellipse} $\mathcal E_\rho \subset \C$ as
\begin{equation*}
\mathcal E_\rho:=\left\{z\in\C\,; z=\tfrac12\big(\rho e^{i\theta}+\rho^{-1} e^{-i\theta}\big)\,, \, 0\leq \theta\leq 2\pi\right\},
\end{equation*}
with foci at $\pm 1$ and major and minor semiaxes of lengths $\frac12(\rho+\rho^{-1})$ and $\frac12(\rho-\rho^{-1})$, respectively;
see Figure \ref{fig:ellipse}.
The following result from \cite{Wa16} will be convenient for our purposes.
\begin{lemma} \label{Lem_Wang}{\rm (Wang \cite[Theorem 4.3]{Wa16})} Let $\mc{K}$ be a function that is analytic inside and on the Bernstein ellipse $\mathcal E_\rho$ for some $\rho >1$. Let $M := \max_{z \in \mathcal E_\rho} |\mc{K}(z)|$. Let $\alpha >0$ and consider the Gegenbauer expansion
$$\mc{K}(t)= \sum_{n=0}^\infty \frak{a}_{n}^{\alpha} \, C_{n}^\alpha(t) \ \ \ (t \in [-1,1]).$$
Then, for any $n \geq 1$, we have the following explicit estimates:
\begin{equation}\label{20210528_16:55}
|\frak{a}_{n}^{\alpha}| \leq 
\left\{
\begin{array}{ll}
\Lambda(n, \rho, \alpha) \left( 1 - \dfrac{1}{\rho^2}\right)^{\alpha -1} \, \dfrac{n^{1 - \alpha}}{\rho^{n +1}}\,,& {\rm if}\ 0 < \alpha \leq 1; \\[0.4cm]
\Lambda(n, \rho, \alpha) \left( 1 + \dfrac{1}{\rho^2}\right)^{\alpha -1} \, \dfrac{n^{1 - \alpha}}{\rho^{n +1}}\,,& {\rm if}\ \alpha > 1,
\end{array}
\right.
\end{equation}
where
\begin{equation}\label{20210601_11:55}
\Lambda(n, \rho, \alpha) := \frac{\Gamma(\alpha)\, M \, \Upsilon_{n}^{1,\alpha}}{\pi}\left(2\left(\rho+\frac1{\rho}\right)+2\left(\frac{\pi}2-1\right)\left(\rho-\frac1{\rho}\right)\right),
\end{equation}
and
\begin{equation}\label{20210601_11:56}
\Upsilon_{n}^{1,\alpha}:=\exp\left(\frac{1-\alpha}{2(n+\alpha-1)}+\frac{1}{12n}\right).
\end{equation}
\end{lemma}
If we succeed in proving that our function $K_g$ admits an analytic continuation past a Bernstein ellipse $\mc{E}_{\rho}$ for some $\rho >1$, then Lemma \ref{Lem_Wang} will be an available tool with $\mathcal{K} = K_g$, $\alpha = \nu$ and $n  = 2\ell$. One readily checks that the bounds provided by \eqref{20210528_16:55} decay exponentially in $n = 2\ell$, and from \eqref{eq:Bounds1} and \eqref{20210528_14:55} we see that it should be possible to achieve \eqref{eq:Compareg1} as long as $M = \max_{z \in \mathcal E_\rho} |K_g(z)|$ is sufficiently small, which ultimately will be verified provided {that the analytic continuation of $\widehat{g}\big|_{\overline{B_4}}$ is sufficiently small in a certain disk. 

\begin{figure}
    \centering
    \includegraphics[width=0.4\textwidth]{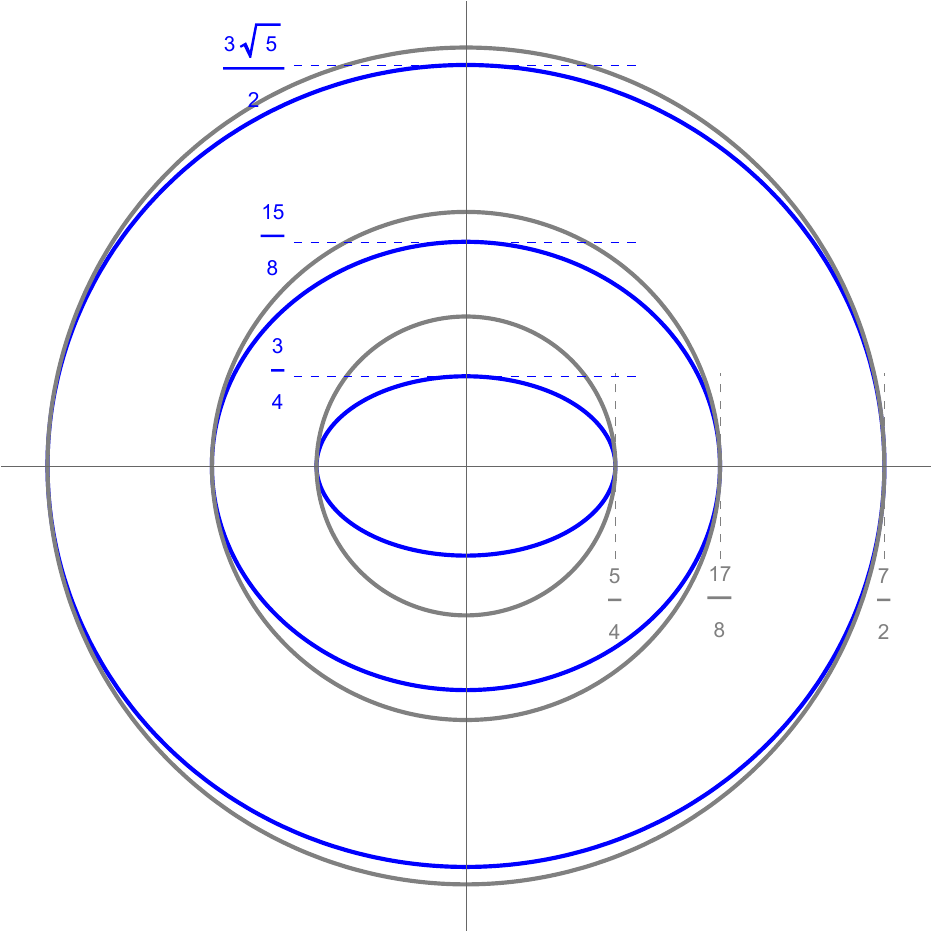}
    \caption{Bernstein ellipses $\mathcal E_\rho$ in the complex plane, $\rho\in\{2,4,\frac{7+3\sqrt 5}2\}$, and the corresponding enveloping disks $\overline{\mc{D}_r}\subset\C$, $r\in\{\frac54,\frac{17}8,\frac72\}$.}
    \label{fig:ellipse}
\end{figure}

\subsection{Analytic continuation of $K_g$} \label{sec:analcontKg} Condition $(\rm R2.A)$ states that $\widehat{g}\big|_{\overline{B_4}}$ admits an analytic extension to a disk $\mc{D}_{R'} \subset \C^d$, with $\overline{\mc{D}_R} \subset \mc{D}_{R'}$ and $R >4$. Recall that, for $t=\omega_1\cdot\omega_2$, we have
\begin{equation*}
K_g(t) = \int_{(\sph{d-1})^2}\widehat g \left(\sum_{j=1}^4\omega_j\right)\left(|\omega_1+\omega_2|^2+|\omega_3+\omega_4|^2-(\omega_1+\omega_2)\cdot(\omega_3+\omega_4)\right) \d\sigma(\omega_3)\,\d\sigma(\omega_4).
\end{equation*}
Set $s:=|\omega_1+\omega_2|=(2+2t)^{\frac12}$ and  $K_g^\star(s):=K_g(t)$. Note that $s \in [0,2]$. We show that $K_g^\star$ can be extended to an {\it even} analytic function on an open disk $\mc{D}_{R''} \subset \C$ of radius $R'' := R'-2 > R-2 > 2$, and hence it admits a power series representation of the form $K_g^\star(s) = \sum_{\ell=0}^{\infty} c_{2\ell} \, s^{2\ell}$, which is absolutely convergent if $|s| < R''$. This plainly implies that $K_g$ can be extended to an analytic function on the open disk $\mc{D}_{R'''} \subset \C$ with $R''' := \big((R'')^2 - 2\big)/2 > 1$ since 
\begin{equation}\label{20210601_12:53}
K_g(t)=K_g^\star(s)
=\sum_{\ell=0}^{\infty} c_{2\ell} \, (s^{2})^{\ell}
=\sum_{\ell=0}^{\infty} c_{2\ell} \, (2 + 2t)^{\ell}
=\sum_{\ell=0}^\infty c'_\ell \,t^\ell.
\end{equation}
This is going to be sufficient for our purposes since the Bernstein ellipse $\mc{E}_{\rho}$ is contained in the closed disk of radius $\tfrac12(\rho + \rho^{-1})$; see Figure \ref{fig:ellipse}. We are then able to choose $\rho>1$ such that $\mc{E}_{\rho} \subset \mc{D}_{R'''}$ . In particular, we can choose $\rho>1$ such that 
\begin{equation}\label{20210603_11:52}
\tfrac12(\rho + \rho^{-1}) = \tfrac12\big((R-2)^2 - 2\big).
\end{equation}

Let us write
\begin{equation}\label{20210601_14:02}
K_g^\star(s)=\textup{I}(s)+\textup{II}(s)-\textup{III}(s)\,,
\end{equation}
where the three summands are defined as follows:
\begin{align}
&\textup{I}(s):=s^2\int_{(\sph{d-1})^2} \widehat g \left(\sum_{j=1}^4\omega_j\right) \d\sigma(\omega_3)\,\d\sigma(\omega_4)\,; \label{20210601_14:03} \\
&\textup{II}(s):=\int_{(\sph{d-1})^2} \widehat g \left(\sum_{j=1}^4\omega_j\right)|\omega_3+\omega_4|^2\, \d\sigma(\omega_3)\,\d\sigma(\omega_4)\,; \label{20210601_14:04}\\
&\textup{III}(s):=\int_{(\sph{d-1})^2}  \widehat g \left(\sum_{j=1}^4\omega_j\right) (\omega_1+\omega_2)\cdot(\omega_3+\omega_4)\,\d\sigma(\omega_3)\,\d\sigma(\omega_4). \label{eq:III}
\end{align}
We show that each of these functions can be extended to an {\it even} analytic function on the open disk $\mc{D}_{R''} \subset \C$. The reasoning for $\textup{I}$ and $\textup{II}$ is similar to that of $\textup{III}$, but simpler. So we focus on $\textup{III}$ only.

\subsubsection{The function $\textup{III}$}\label{Sec6.2.1}
Recall that coordinates for $\omega\in\sph{d-1}$ can be defined recursively:
\begin{equation*}
\omega=(\zeta\sin\theta_{d-1},\cos\theta_{d-1}),\quad \zeta\in\sph{d-2},
\end{equation*}
with 
\begin{equation*}
\d\sigma_{d-1}(\omega) =(\sin \theta_{d-1})^{d-2} \,\d\theta_{d-1}\,\d\sigma_{d-2}(\zeta)\,,
\end{equation*}
where we denote by $\sigma_{d-j}$ the surface measure on the unit sphere $\sph{d-j}$. Since the arc length measure on $\sph{1}$ is simply $\d\theta_1$, it follows by induction that
\begin{equation*}
\d\sigma_{d-1}=\prod_{j=1}^{d-2}(\sin\theta_{d-j})^{d-j-1}\,\d\theta_{d-j}\,\d\theta_1\,,
\end{equation*} 
where $0\leq\theta_1\leq2\pi$ and $0\leq\theta_j\leq \pi$ for $j\in\{2,3,\ldots,d-1\}$.
Going back to \eqref{eq:III}, by the radiality of $\widehat g|_{\overline{B_4}}$, no generality is lost in assuming that $\omega_1+\omega_2=s\,e_1$ for some $s\in[0,2]$,
where $e_1=(1,0,\ldots,0)\in\R^d$ denotes the first coordinate vector. Writing $x=(x_1,x')\in\R\times\R^{d-1}$, we have that
\begin{align}\label{eq:PrepareIII}
\begin{split}
\textup{III}(s)=\int_{[0,2\pi]^2}\int_{[0,\pi]^{2d-4}} &\widehat g\left(s+(\omega_3+\omega_4)_1,(\omega_3+\omega_4)'\right)
s\,(\omega_3+\omega_4)_1 \\
& \qquad \qquad \times\prod_{j=1}^{d-2}(\sin\theta_{d-j})^{d-j-1}\d\theta_{d-j}
\prod_{j=1}^{d-2}(\sin\tilde\theta_{d-j})^{d-j-1}\d\tilde\theta_{d-j}\,\d\theta_1\,\d\tilde\theta_1\,,
\end{split}
\end{align}
where the variables of integration $\omega_3$ and $\omega_4$ are coordinate-wise given 
as follows:
\begin{align}\label{20210720_23:13}
\begin{split}
\omega_{3,1}=\prod_{j=1}^{d-1} \sin\theta_j, &\quad \omega_{4,1}=\prod_{j=1}^{d-1} \sin\tilde\theta_j\,;\\
\omega_{3,2}=\cos\theta_1\prod_{j=2}^{d-1} \sin\theta_j, &\quad\omega_{4,2}=\cos\tilde\theta_1\prod_{j=2}^{d-1} \sin\tilde\theta_j\,;\\ 
\vdots&\quad\vdots\\
\omega_{3,d-1}=\cos\theta_{d-2} \sin\theta_{d-1}, &\quad\omega_{4,d-1}=\cos\tilde\theta_{d-2} \sin\tilde\theta_{d-1}\,;\\
\omega_{3,d}=\cos\theta_{d-1},&\quad \omega_{4,d}=\cos\tilde\theta_{d-1}.
\end{split}
\end{align}

Note that \eqref{eq:PrepareIII} can be used to extend the domain of definition of the function $\textup{III}$ to $\mc{D}_{R''} \subset \C$, by replacing $\widehat g|_{\overline{B_4}}$ by its analytic continuation $\mc{G}$. Such an extended function $s\mapsto \textup{III}(s)$ is clearly continuous.
Let $\gamma$ be an arbitrary\footnote{``Triangle'' would suffice.} closed piecewise $C^1$-curve in $\mc{D}_{R''} \subset \C$.
Then from \eqref{eq:PrepareIII} and Fubini's Theorem it follows that
\begin{align}\label{20210531_12:50}
\begin{split}
\int_\gamma \textup{III}(s)\,\d s
=\int_{[0,2\pi]^2}\int_{[0,\pi]^{2d-4}} & (\omega_3+\omega_4)_1 
\left(\int_\gamma {\mc{G}}\left(s+(\omega_3+\omega_4)_1,(\omega_3+\omega_4)'\right)s\,\d s\right)
\\
&\times\prod_{j=1}^{d-2}(\sin\theta_{d-j})^{d-j-1}\d\theta_{d-j}
\prod_{j=1}^{d-2}(\sin\tilde\theta_{d-j})^{d-j-1}\d\tilde\theta_{d-j}\,\d\theta_1\,\d\tilde\theta_1=0.
\end{split}
\end{align}
Indeed, the innermost integral on the right-hand side of \eqref{20210531_12:50} vanishes by Cauchy's Theorem, since the function {$\mc{G}$} is analytic on $\mc{D}_{R'} \subset \C^d$ (in particular, in its first coordinate). By Morera's Theorem, {it then follows that $\textup{III}$ defines an analytic function on $\mc{D}_{R''} \subset \C$.} {Finally, observe in \eqref{eq:PrepareIII}} that the change of variables $(\theta_1,\tilde\theta_1)\mapsto(-\theta_1,-\tilde\theta_1) \mod 2\pi$, corresponding to a reflection across the hyperplane $\langle e_1\rangle^\perp$, and the radiality of $\widehat g|_{\overline{B_4}}$ together reveal that $\textup{III}(-s)=\textup{III}(s)$, {first for every $s \in (-2,2)$, and hence for every $s\in \mc{D}_{R''}$}.

\smallskip

This yields the qualitative proof of Lemma \ref{Lem6_crux}, and we now proceed to the effective implementation.

{\subsection{Auxiliary integrals} Let us record two integrals that shall be relevant for the upcoming discussion.
\begin{lemma}
Let $\zeta \in \sph{d-1}$ be given. We have:
\begin{align}
\int_{(\sph{d-1})^2} |\omega_3+\omega_4|^2\, \d\sigma(\omega_3)\,\d\sigma(\omega_4) & = 2 \, \sigma(\sph{d-1})^2\,; \label{20210720_23:08}\\
\int_{(\sph{d-1})^2} \big|(\omega_3+\omega_4)\cdot \zeta\big|\, \d\sigma(\omega_3)\,\d\sigma(\omega_4) & = \left(\frac{2^{d-1}\,\Gamma\left(\frac{d}{2}\right)^3}{\pi  \, \Gamma\big( d - \frac12)\,\Gamma\left(\frac{d+1}{2}\right)}\right)   \sigma(\sph{d-1})^2 =: \frak{r}_d \, \sigma(\sph{d-1})^2.\label{20210720_23:09}
\end{align}
\end{lemma}
\noindent {\sc Remark:} For our purposes, the pertinent values of the constants $\frak{r}_d$ are: 
\begin{equation*}
\frak{r}_3 = \frac{2}{3}  \ \ ; \ \  \frak{r}_4 =  \frac{2^8}{45\,\pi^2} \ \ ; \ \  \frak{r}_5= \frac{18}{35}\ \ ; \ \  \frak{r}_6 = \frac{2^{16}}{14175 \,\pi^2} \ \ ; \ \  \frak{r}_7 = \frac{100}{231}.
\end{equation*}
\begin{proof}
Identity \eqref{20210720_23:08} follows simply from the relation $|\omega_3+\omega_4|^2 = 2 + 2 \omega_3 \cdot \omega_4$, together with the fact that $\int_{(\sph{d-1})^2} (\omega_3\cdot \omega_4)\, \d\sigma(\omega_3)\,\d\sigma(\omega_4) = 0$. The proof of \eqref{20210720_23:09} is more interesting. First notice that the left-hand side of \eqref{20210720_23:09} is independent of $\zeta \in \sph{d-1}$ and hence we may assume without loss of generality that $\zeta = e_d = (0,0, \ldots, 1)$. Recall from \cite[Lemma 5]{COS15} the exact expression for the two-fold convolution of the surface measure $\sigma$:
\begin{align}\label{eq:twofoldconv}
(\sigma*\sigma)(x) = 2^{3-d}\, \sigma_{d-2}(\sph{d-2})\, \frac{(4 - |x|^2)^{\frac{d-3}{2}}}{|x|} \, {\bf 1}_{B_2}(x) \ \ \ ; \ \ \ (x \in \R^d).
\end{align}
The left-hand side of \eqref{20210720_23:09} is equal to 
\begin{align*}
& \int_{(\sph{d-1})^2} \int_{\R^d}  \boldsymbol{\delta}(x - \omega_3-\omega_4) \, |x\cdot e_d| \, \dx \, \d\sigma(\omega_3)\,\d\sigma(\omega_4) =  \int_{\R^d} (\sigma*\sigma)(x)\,  |x\cdot e_d|\, \dx\\
& = 2^{3-d}\, \sigma_{d-2}(\sph{d-2}) \int_0^2  \left(\int_{\sph{d-1}} |\omega\cdot e_d|\,\d\sigma(\omega)\right) (4 - r^2)^{\frac{d-3}{2}}\, r^{d-1} \,\d r\\
& = 2^{3-d}\, \big(\sigma_{d-2}(\sph{d-2})\big)^2  \left(\int_0^{\pi} |\cos \theta|\,(\sin \theta)^{d-2}\d\theta\right) \left(\int_0^2 (4 - r^2)^{\frac{d-3}{2}}\, r^{d-1} \,\d r\right)\\
& = 2^{3-d}\, \big(\sigma_{d-2}(\sph{d-2})\big)^2 \left( \frac{2}{d-1}\right)\left( 2^{2d-4} \int_0^1 (1-s)^{\frac{d-3}{2}}\, s^{\frac{d-2}{2}}\,\d s\right)\\
&= \frac{2^d\,\big(\sigma_{d-2}(\sph{d-2})\big)^2}{d-1} \, \frac{\Gamma\left(\frac{d-1}{2}\right)\,\Gamma\left(\frac{d}{2}\right)}{\Gamma\left(d - \frac{1}{2}\right)}\,,
\end{align*}
and the latter is equal to the right-hand side of \eqref{20210720_23:09}. Here, in the second identity we changed variables to polar coordinates $x = r\omega$, in the third identity we changed the variable $\omega$ as described in \eqref{20210720_23:13}, in the fourth identity we evaluated the trigonometric integral and changed variables $r^2 = 4s$ in the other integral, and in the fifth identity we used the Beta function evaluation $\int_0^1 (1-s)^{a-1} s^{b-1}\,\d s = \frac{\Gamma(a)\,\Gamma(b)}{\Gamma(a+b)}$ for $a,b>0$. 
\end{proof}}

\subsection{Quantifying the perturbation} In our setup, recall that $\nu := \tfrac{d-2}{2} \in \big\{\tfrac12, 1, \tfrac32, 2, \tfrac52\big\}$. Our objective now is to bound \eqref{20210528_14:55} using Lemma \ref{Lem_Wang}. We choose the particular $\rho >1$ given by \eqref{20210603_11:52}, verifying
\begin{equation}\label{20210729_14:06}
\rho + \rho^{-1} = (R-2)^2 - 2.
\end{equation}
From \eqref{20210601_11:55} and \eqref{20210601_11:56} we obtain
\begin{equation}\label{20210603_12:15}
\Lambda(2\ell, \rho, \nu) = \frac{\Gamma(\nu)}{\pi} \exp\left(\frac{1-\nu}{2(2\ell+\nu-1)}+\frac{1}{24\ell}\right) \left(2\left(\rho+\frac1{\rho}\right)+2\left(\frac{\pi}2-1\right)\left(\rho-\frac1{\rho}\right)\right) \max_{w\in\mathcal E_{\rho}}\left\lvert K_g(w)\right\rvert.
\end{equation}
{Let $\mc{M}_{d,R} := \max_{z \in \overline{\mc{D}_R}} |\mc{G}(z)|$}. At this point, we want to bound $\max_{w \in \mathcal E_{\rho}}\left\lvert K_g(w)\right\rvert$ in terms of {$\mc{M}_{d,R}$}. We have seen in \eqref{20210601_12:53} that, via the change of variables $w = (s^2 - 2)/2$, we have $K_g(w)=K_g^\star(s)$, and whenever $w \in \mathcal E_{\rho}$ we have $s \in \overline{\mc{D}_{R-2}}$. Using \eqref{20210601_14:02}--\eqref{eq:III}, it follows that
\begin{align}\label{20210601_14:41}
|K_g(w)|=|K_g^\star(s)|\leq |\textup{I}(s)|+|\textup{II}(s)|+|\textup{III}(s)|.
\end{align}
Regarding $\textup{III}(s)$, we look at it via \eqref{eq:PrepareIII}, yielding the analytic continuation. Using the elementary estimates 
\begin{align}\label{20210713_11:03}
 |s| \leq R-2 \ \ \ {\rm and} \ \ \ |s\,e_1 + \omega_3 + \omega_4| \leq |s| + |\omega_3 + \omega_4| \leq R\,,
\end{align}
together with \eqref{20210720_23:09}, we plainly get from definition~\eqref{eq:III}
\begin{align}
|\textup{III}(s)| & \leq (R-2) \,\mc{M}_{d,R} \int_{(\sph{d-1})^2}  \big|(\omega_3+\omega_4)_1\big|\,\d\sigma(\omega_3)\,\d\sigma(\omega_4) \nonumber \\
&\leq   (R-2) \,\frak{r}_d\, \sigma(\sph{d-1})^2\,\mc{M}_{d,R}.\label{20210601_14:42}
\end{align}
Similarly, using the analogous expressions for the analytic continuations of $\textup{I}(s)$ and $\textup{II}(s)$, the elementary inequalities \eqref{20210713_11:03}, and identity \eqref{20210720_23:08} for $\textup{II}(s)$, one finds
\begin{equation}\label{20210601_14:43}
|\textup{I}(s)| \leq  (R-2)^2\, \sigma(\sph{d-1})^2 \, \mc{M}_{d,R} \ \ \ {\rm and} \ \ \ |\textup{II}(s)| \leq    2\,\sigma(\sph{d-1})^2\, \mc{M}_{d,R}.
\end{equation}
Putting together \eqref{20210601_14:41}, \eqref{20210601_14:42} and \eqref{20210601_14:43} we find that 
\begin{align}\label{20210601_15:34}
\max_{w\in\mathcal E_\rho}\left\lvert K_g(w)\right\rvert \leq  {\big((R-2)^2 + (R-2)\frak{r}_d + 2\big)\, \sigma(\sph{d-1})^2\,\mc{M}_{d,R}.}
\end{align}
From \eqref{20210528_14:55},  \eqref{20210528_16:55}, \eqref{20210603_12:15} and \eqref{20210601_15:34}, we obtain
\begin{equation}\label{20210603_12:37}
|\lambda_{d,g}(2\ell)| \leq \beta_{d,R} \,G_{d,R}(\ell)\,\mc{M}_{d,R}\,,
\end{equation}
with the constant $\beta_{d,R}$ given by 
\begin{equation*}
\beta_{d,R}:= 2^{2-\nu}\, \pi^{\nu}  \left[2\left(\rho+\frac1{\rho}\right)+2\left(\frac{\pi}2-1\right)\left(\rho-\frac1{\rho}\right)\right] \left( 1 \pm \dfrac{1}{\rho^2}\right)^{\nu -1} {\big((R-2)^2 + (R-2)\frak{r}_d + 2\big)\, \sigma(\sph{d-1})^2}
\end{equation*}
(the minus sign above is used for $\nu \in \big\{\tfrac12, 1\big\}$ and the plus sign for $\nu \in \big\{\tfrac32, 2, \tfrac52\big\}$) and 
$$G_{d,R}(\ell) := \dfrac{1}{(2\ell + \nu)}\, \exp\left(\dfrac{1-\nu}{2(2\ell+\nu-1)}\!+\!\dfrac{1}{24\ell}\right)\dfrac{\ell^{1 - \nu}}{\rho^{2\ell +1}}.$$

\subsection{Final comparison} Let us write the bounds on the right-hand sides of \eqref{20210601_09:57}--\eqref{20210601_09:61} as $- c_d\,\ell^{-d}$ (i.e. $c_3 = \frac{2^8 \pi^5}{35}$, and so on). Hence, from \eqref{eq:Compareg1}, \eqref{20210601_09:57}--\eqref{20210601_09:61}, and \eqref{20210603_12:37}, it suffices to have that 
\begin{equation*}
\beta_{d,R} \,G_{d,R}(\ell)\,\mc{M}_{d,R} < c_d\,\ell^{-d}.
\end{equation*}
Since this must hold for every $\ell \in \N:=\{1,2,3,\ldots\}$,  equivalently we have to ensure that  
\begin{equation}\label{20210729_13:07}
\mc{M}_{d,R}< \frac{c_d}{\beta_{d,R} \left(\displaystyle\max_{\ell \in \N} G_{d,R}(\ell) \, \ell^d\right)}=:\mathcal{A}_{d, R}.
\end{equation}
Note that the non-negative function $F_{d,R}(\ell) := G_{d,R}(\ell) \, \ell^d$ is exponentially {decaying in $\ell$} (since $\rho >1$), and hence it must attain its maximum value (over $\N$) at a certain $\ell^*_{d,R}$. For instance, a routine verification yields the following values:
\begin{center}
\begin{tabular}{ |l|c|c|c|c|c|c|c|c|c|c| } 
 \hline
 & $\!\ell^*_{3,R}\!$ & $\!\ell^*_{4,R}\!$ & $\!\ell^*_{5,R}\!$& $\!\ell^*_{6,R}\!$& $\!\ell^*_{7,R}\!$ & $\mc{A}_{3,R}$ & $\mc{A}_{4,R}$ & $\mc{A}_{5,R}$  & $\mc{A}_{6,R}$ & $\mc{A}_{7,R}$\\ 
 \hline
 $\!R =  \frac{3\sqrt{2}}{2} +2$  ;  ($\rho = 2$)\! \!\! & $2$ & $2$ & $3$ & $3$ & $4$  & $ {0.534\ldots}\!$ & ${5.064\ldots}\!$ & ${10.276 \ldots}\!$ & ${14.576\ldots}\!$ & ${13.745  \ldots} \!$ \\
  \hline
 $\!R =  \tfrac92$  ;  ($\rho = 4$) \!\!\! & $1$ & $1$ & $1$ & $2$ & $2$  & ${2.805\ldots}\!$ & ${39.860 \ldots}\!$ & ${157.795 \ldots}\!$ & ${333.547 \ldots}\!$ & $ {430.015 \ldots} \!$\\ \hline
  $\!R =  5$  ;  \big($\rho = \frac{7 + 3\sqrt{5}}{2})$\big)\! \!\! & $1$ & $1$ & $1$ & $1$ & $1$  & ${6.493\ldots}\!$ & ${90.100\ldots}\!$ & ${363.294\ldots}\!$ & ${1092.176\ldots}\!$ & $ {2131.265\ldots} \!$\\ \hline
\end{tabular}
\end{center}

\bigskip

This concludes the proof of Lemma \ref{Lem6_crux} in the analytic case.

\subsection{Limiting behavior} {When $R$ is sufficiently large, one can easily verify that $\ell^*_{d,R}= 1$. Hence,
$$\mathcal{A}_{d, R} = \frac{c_d}{\beta_{d,R} \, G_{d,R}(1)}.$$
Recalling \eqref{20210729_14:06}, this plainly implies that 
\begin{equation}\label{20210729_14:34}
\mc{L}_d := \lim_{R \to \infty} \frac{\mc{A}_{d,R}}{R^2} = \frac{c_d}{2^{2- \nu}\, \pi^{\nu +1}\,\sigma(\sph{d-1})^2\, \dfrac{1}{(2 + \nu)}\, \exp\left(\dfrac{1-\nu}{2(1+\nu)}\!+\!\dfrac{1}{24}\right)}\, .
\end{equation}
The corresponding evaluation yields 
$$\mc{L}_3 = 1.826\ldots \ ; \ \mc{L}_4 =24.555\ldots  \ ; \ \mc{L}_5=98.593\ldots  \ ; \  \mc{L}_6 = 296.255\ldots \ ; \ \mc{L}_7 = 579.209\ldots $$
and this concludes the proof of Theorem \ref{Thm1}.}

\section{Proof of Lemma \ref{Lem6_crux}: $C^k$-version}
{In this section we work under the hypotheses of Theorem \ref{Thm1b}; in particular, (R2.C) holds. Recall  $\nu := \tfrac{d-2}{2}$}. 

\subsection{Bounds for the Gegenbauer coefficients ($C^k$-version)} Not having found in the literature a result that would exactly fit our purposes (like Lemma \ref{Lem_Wang} in the analytic case),  we briefly work our way up from first principles. Recall the value of $C_{n}^\alpha(1)$ and the definition of the constant $\frak{h}_{n}^{\alpha}$ in \eqref{eq:ValueAt1}. We start with the following lemma.

\begin{lemma} \label{Lem_Wang_Sub} Let $\mc{K} \in C^0[-1,1] \cap C^k(-1,1)$ for some $k \geq 1$. Let $\alpha >0$ and assume further that \begin{align}\label{20210713_18:20}
t \mapsto \mc{K}^{(j)}(t) \, (1-t^2)^{\alpha + j - \frac12} \ \ {\it is \ bounded \ in} \ (-1,1) \ {\it for} \, 1 \leq j \leq k.
\end{align}
Consider the Gegenbauer expansion
\begin{align}\label{20210712_15:03}
\mc{K}(t)= \sum_{n=0}^\infty \frak{a}_{n}^{\alpha} \, C_{n}^\alpha(t).
\end{align}
Let $2 \leq p \leq \infty$. Then, for any $n \geq k$, we have the following estimate:
\begin{align}\label{20210712_16:52}
|\frak{a}_{n}^{\alpha}| \leq \frac{D_{n,k}^{\alpha}  \left(C_{n-k}^{\alpha+k}(1)\right)^{1 - \frac{2}{p}} \big(\frak{h}_{n-k}^{\alpha+k} \big)^{\frac{2}{p}}}{(\frak{h}_{n}^{\alpha})^2 } \left(\int_{-1}^1 \big|\mc{K}^{(k)}(t)\big|^{p'} (1 - t^2)^{\alpha +k -\frac12}\,\dt \right)^{\frac{1}{p'}} \,,
\end{align}
where $\frac{1}{p} + \frac{1}{p'}=1$, and 
\begin{equation}\label{20210714_10:38}
D_{n,k}^{\alpha} := \prod_{j=0}^{k-1} \frac{2(\alpha + j)}{(n-j)(n + 2\alpha + j)}.
\end{equation}
\end{lemma}
\begin{proof} Recall the indefinite integral \cite[Eq. 22.13.2]{AS},
\begin{equation}\label{20210607_10:131}
\int C_n^{\alpha}(t) (1 - t^2)^{\alpha - \frac12}\,\dt =  \frac{- 2\alpha}{n (n + 2\alpha)} \, C_{n-1}^{\alpha+1}(t) (1 - t^2)^{\alpha + \frac12}. 
\end{equation}
From \eqref{20210712_15:03}, we may apply integration by parts $k$ times, using \eqref{20210607_10:131} and \eqref{20210713_18:20} (to eliminate the boundary terms at each iteration\footnote{Condition \eqref{20210713_18:20} can be weakened, but its present form suffices for our purposes.}), to get
\begin{align}\label{20210712_15:12}
\frak{a}_{n}^{\alpha} \,(\frak{h}_{n}^{\alpha})^2   =\int_{-1}^1 \mc{K}(t)\,C_{n}^\alpha(t) \,(1 - t^2)^{\alpha - \frac12}\,\dt = D_{n,k}^{\alpha} \int_{-1}^1 \mc{K}^{(k)}(t)\,C_{n-k}^{\alpha+k}(t) \,(1 - t^2)^{\alpha +k -\frac12}\,\dt.
\end{align}
Let $\delta = 1 - \frac{2}{p} \geq 0$.\footnote{{This is where the hypothesis $p\geq 2$ is needed: in order to have $\delta \geq 0$ and, consequently, the valid inequality \eqref{20210806_11:35}.}} Using that 
\begin{equation}\label{20210806_11:35}
\left|C_{n-k}^{\alpha+k}(t)\right| \leq \left(C_{n-k}^{\alpha+k}(1)\right)^{\delta}\, \left|C_{n-k}^{\alpha+k}(t)\right|^{1-\delta}\,,
\end{equation}
and applying H\"{o}lder's inequality with exponents $p$ and $p'$ below, we observe that the right-hand side of \eqref{20210712_15:12} is, in absolute value, dominated by
\begin{align*}
& \leq D_{n,k}^{\alpha}  \left(C_{n-k}^{\alpha+k}(1)\right)^{\delta}\int_{-1}^1 \left( \big|\mc{K}^{(k)}(t)\big| (1 - t^2)^{(\alpha +k -\frac12)(\frac{1 + \delta}{2})} \right) \,\left(\left|C_{n-k}^{\alpha+k}(t)\right|^{1-\delta}  \, (1 - t^2)^{(\alpha +k -\frac12)(\frac{1 - \delta}{2})}\right)\,\dt\\
& \leq D_{n,k}^{\alpha}  \left(C_{n-k}^{\alpha+k}(1)\right)^{\delta}\left(\int_{-1}^1 \big|\mc{K}^{(k)}(t)\big|^{p'} (1 - t^2)^{\alpha +k -\frac12}\,\dt \right)^{\frac{1}{p'}}  \,\left( \int_{-1}^1C_{n-k}^{\alpha+k}(t)^{2}  \, (1 - t^2)^{\alpha +k -\frac12}\,\dt\right)^{\frac{1}{p}}\\
& = D_{n,k}^{\alpha}  \left(C_{n-k}^{\alpha+k}(1)\right)^{\delta} \big(\frak{h}_{n-k}^{\alpha+k} \big)^{\frac{2}{p}}\left(\int_{-1}^1 \big|\mc{K}^{(k)}(t)\big|^{p'} (1 - t^2)^{\alpha +k -\frac12}\,\dt \right)^{\frac{1}{p'}}.
\end{align*}
This yields the proposed estimate.
\end{proof}
\noindent{\sc Remark}: Observe that in Lemma \ref{Lem_Wang_Sub} we are not specializing to $k = k(d) = \lfloor (d+3)/2 \rfloor$;
rather, it holds for any $k \geq 1$. Further observe that, for fixed $k$, as $n \to \infty$, we have $D_{n,k}^{\alpha} \simeq n^{-2k}$; 
\ $\frak{h}_{n}^{\alpha} \simeq n^{\alpha -1}$; \ $\frak{h}_{n-k}^{\alpha+k} \simeq n^{\alpha + k -1}$;\ $C_{n-k}^{\alpha+k}(1) \simeq n^{2\alpha + 2k -1}$. Hence, assuming that the integral on the right-hand side of \eqref{20210712_16:52} is finite, the dependence on $n$ of \eqref{20210712_16:52} is given by
$$\frac{D_{n,k}^{\alpha}  \left(C_{n-k}^{\alpha+k}(1)\right)^{1 - \frac{2}{p}} \big(\frak{h}_{n-k}^{\alpha+k} \big)^{\frac{2}{p}}}{(\frak{h}_{n}^{\alpha})^2 } \simeq n^{-2k + ({1 - \frac{2}{p}})(2\alpha +2k -1) + \frac{2}{p}(\alpha + k -1) - 2\alpha + 2}.$$
From \eqref{eq:Compareg1} and \eqref{20210528_14:55}, this decay in $n$ (with $\alpha = \nu)$ will suffice for our purposes, provided that 
$$-2k + \big({1 - \tfrac{2}{p}}\big)(2\nu +2k -1) + \tfrac{2}{p}(\nu + k -1) - 2\nu + 2 \leq  -d+1.$$
{This leads us to 
\begin{equation}\label{20210805_10:57}
k \geq \frac{d(p-1) +2}{2}.
\end{equation}
Since $p \geq 2$, inequality \eqref{20210805_10:57} plainly implies that} $k \geq (d+2)/2$ and, since $k$ is an integer, we end up with $k  \geq \lfloor (d+3)/2 \rfloor$ as our minimal regularity assumption in this setup.

\smallskip

From now on we specialize matters to our particular situation by letting, in the notation of Lemma \ref{Lem_Wang_Sub}, 
\begin{equation}\label{20210713_12:36}
\mc{K} = K_g;\ \  \alpha = \nu = \frac{d-2}{2};\ \  k = k(d) = \left\lfloor \frac{d+3}{2} \right\rfloor; \ \  {\rm and} \ p = p(d) = 2 + \left(\frac{1 - (-1)^d}{2}\right)\frac{1}{d}.
\end{equation}
We postpone the discussion of why the function $K_g$ verifies condition \eqref{20210713_18:20} until the next subsection, and for now follow up with a suitable upper bound for the integral appearing on the right-hand side of \eqref{20210712_16:52} .

\begin{lemma} \label{Lem_Wang_Sub_2}
Let $d\in\{3,4,5,6,7\}$. In the notation of Lemma \ref{Lem_Wang_Sub}, with the specialization \eqref{20210713_12:36}, we have
\begin{align*}
\left(\int_{-1}^1 \big|K_g^{(k)}(t)\big|^{p'} (1 - t^2)^{(\nu +k -\frac12)}\,\dt \right)^{\frac{1}{p'}} \leq  \, 2^{\frac{1}{p'}}  \sup_{t \in (-1,1)} \left( \left| K_g^{(k)}(t)\right| \,(2 + 2t)^{k-1}\right). 
\end{align*}
\end{lemma}
\begin{proof} Observe that, for $t \in (-1,1)$,  
 \begin{align}\label{20210713_13:01}
 \begin{split}
 \left| K_g^{(k)}(t)\right| \,(1 - t^2)^{\frac{1}{p'}(\nu +k -\frac12)}& =   \left| K_g^{(k)}(t)\right| \,(1 + t)^{\frac{1}{p'}(\nu +k -\frac12)}\,\,(1 - t)^{\frac{1}{p'}(\nu +k -\frac12)}\\
 & \leq 2^{\frac{1}{p'}(\nu +k -\frac12)} \left| K_g^{(k)}(t)\right| \,(1 + t)^{\frac{1}{p'}(\nu +k -\frac12)}\,\\
 & =  \left| K_g^{(k)}(t)\right| \,(2 + 2t)^{\frac{1}{p'}(\nu +k -\frac12)}\\
 & = \left( \left| K_g^{(k)}(t)\right| \,(2 + 2t)^{k-1}\right) (2 + 2t)^{\frac{1}{p'}(\nu +k -\frac12) - (k-1)}\\
 & \leq  \left[\sup_{t \in (-1,1)}\left( \left| K_g^{(k)}(t)\right| \,(2 + 2t)^{k-1}\right) \right](2 + 2t)^{\frac{1}{p'}(\nu +k -\frac12) - (k-1)}.
 \end{split}
 \end{align}
Note that with the specialization \eqref{20210713_12:36} we have
 \begin{align}\label{20210713_14:36}
 \left(\nu +k -\frac12\right) - p'(k-1) = -\frac{1}{2}.
 \end{align}
 Hence, from \eqref{20210713_13:01} and \eqref{20210713_14:36} we plainly get
 \begin{align}\label{20210806_11:20}
 \left(\int_{-1}^1 \big|K_g^{(k)}(t)\big|^{p'} (1 - t^2)^{\nu +k -\frac12}\,\dt \right)^{\frac{1}{p'}}  \leq  \left[\sup_{t \in (-1,1)} \left(\big| K_g^{(k)}(t)\big| \,(2 + 2t)^{k-1}\right)\right]\left(\int_{-1}^1 (2 + 2t)^{-\frac12}\,\dt \right)^{\frac{1}{p'}}\,,
 \end{align}
which leads us to the desired conclusion.
\end{proof}
\noindent{{\sc Remark:} There is a subtle reason for the particular choice of $p(d)$ in \eqref{20210713_12:36}. The reader may wonder why we are not simply choosing $p=2$ in all cases. The reason is as follows. There are two competing forces for the value of $p$ in our argument. On the one hand, from \eqref{20210805_10:57} one sees that, the smaller the value of $p$, the smaller the number of derivatives we have to require from our function (which we intend to keep to a minimum). On the other hand, larger values of $p$ place us in a better position to control potential singularities arising in the proof of Lemma \ref{Lem_Wang_Sub_2}, a crucial intermediate step in our proof. When the dimension $d$ is even, the choice $p=2$ yields an integer number on the right-hand side of \eqref{20210805_10:57} and we proceed with this choice. When the dimension $d$ is odd, the choice $p=2$ yields an integer plus a half on the right-hand side of \eqref{20210805_10:57}. Since we are not entering the realm of fractional derivatives in this paper, this would force us to move $k$ to the next integer (as such, in some vague sense, we would have half a derivative to spare). Moreover, for odd dimensions $d$, such a choice $p=2$ and $k = \lceil (d+2)/2\rceil = (d+3)/2$ would yield exactly $-1$ in place of $-\frac{1}{2}$ on the right-hand side of \eqref{20210713_14:36}, which in turn would make the corresponding integral on the right-hand side of \eqref{20210806_11:20} diverge. The natural solution is then to use this spare half derivative to increase the value of $p$ slightly, making the right-hand side of \eqref{20210713_12:36} coincide with the integer $(d+3)/2$. This leads us to the choice $p(d)$.}

\subsection{Relating the derivatives of $K_g$ and $K_g^{\star}$} As in \S \ref{sec:analcontKg}, set $s:=(2+2t)^{\frac12}$ and  $K_g^\star(s):=K_g(t)$, with $t \in [-1,1]$ and $s \in [0,2]$. The next task is to express the derivatives of $K_g(t)$ in terms of derivatives of $K_g^{\star}(s)$. We collect the relevant information in the next lemma.
\begin{lemma} \label{Lem12_July13} {Assume that $K_g^\star:(0,2) \to \R$ is sufficiently smooth.} For $t \in (-1,1)$ we have
\begin{align*}
K_g^{(1)}(t) &= (K_g^\star)^{(1)}(s)\,s^{-1}\,;\\
K_g^{(2)}(t)\,(2 + 2t) &= (K_g^\star)^{(2)}(s) -(K_g^\star)^{(1)}(s)\,s^{-1}\,;\\
K_g^{(3)}(t)\,(2 + 2t)^2 &= (K_g^\star)^{(3)}(s)\,s -3(K_g^\star)^{(2)}(s) + 3(K_g^\star)^{(1)}(s) \, s^{-1}\,;\\
K_g^{(4)}(t)\,(2 + 2t)^3 &= (K_g^\star)^{(4)}(s)\,s^2 -6(K_g^\star)^{(3)}(s)\,s + 15(K_g^\star)^{(2)}(s) - 15(K_g^\star)^{(1)}(s)\,s^{-1}\,;\\
K_g^{(5)}(t)\,(2 + 2t)^4 &= (K_g^\star)^{(5)}(s)\,s^3 -10(K_g^\star)^{(4)}(s)\,s^2 + 45(K_g^\star)^{(3)}(s)\,s - 105(K_g^\star)^{(2)}(s) +105(K_g^\star)^{(1)}(s) \,s^{-1}.
\end{align*}
\end{lemma}
\begin{proof}
Note that, for $t \in (-1,1)$, we have $s \in (0,2)$ and $\frac{\partial s}{\partial t} = \frac{1}{s}$. Hence
$$ \frac{\partial }{\partial t} K_g(t) = \left(\frac{1}{s} \frac{\partial }{\partial s} \right)K_g^{\star}(s).$$
The lemma follows by applying the operator $\left(\frac{1}{s} \frac{\partial }{\partial s} \right)$ to $K_g^{\star}(s)$ a total of $j$ times ($1 \leq j \leq 5$), and then multiplying by $(2 + 2t)^{j-1} = s^{2j-2}$. 
\end{proof}
Recall the representation \eqref{20210601_14:02}--\eqref{eq:III} for $K_g^{\star}(s)$, after the change of variables proposed in \eqref{eq:PrepareIII} (which is performed on  $\textup{III}(s)$ but applies to $\textup{I}(s)$ and $\textup{II}(s)$ as well). Note that the function $\widehat{g}\big|_{\overline{B_4}}$ appears in this expression and its regularity now enters into play. Since $\widehat{g} \in C^{k}(B_4) \cap C^0(\overline{B_4})$, with bounded partial derivatives of order up to $k$ in $B_4$ (condition (R2.C)), expression  \eqref{eq:PrepareIII} and its analogues for $\textup{I}(s)$ and $\textup{II}(s)$ define $K_g^{\star}(s)$ as an even function that belongs $C^k((-2,2)) \cap C^0([-2,2])$, with bounded partial derivatives of order up to $k$ in $(-2,2)$ (for the claim that it is {\it even}, the argument is as in \S \ref{Sec6.2.1}). In particular $(K_g^\star)^{(1)}(0) = 0$ and the mean value theorem yields, for any $s \in (0,2)$, 
\begin{align}\label{20210713_17:13}
\big|(K_g^\star)^{(1)}(s)\big| \leq |s| \max_{{u \in [0,s]}} \left| (K_g^\star)^{(2)}(u)\right|.
\end{align}
As a by-product of Lemma \ref{Lem12_July13} observe that all the functions $K_g^{(j)}(t)\,(2 + 2t)^{j-1}$ ($1 \leq j \leq 5$) are bounded in $(-1,1)$ and hence condition \eqref{20210713_18:20} clearly holds. Our next result bounds the expressions $(K_g^\star)^{(j)}(s)\,s^{j-2}$ ($1 \leq j \leq 5$), appearing in Lemma \ref{Lem12_July13}, in terms of  the supremum of the partial derivatives of $\widehat{g}$. 

{\begin{lemma}\label{Lem13_July13}
Let $d \in \{3,4,5,6,7\}$ and $\frak{r}_d$ as in \eqref{20210720_23:09}. Let {${\mc M}_d = \max_{\alpha} \sup_{\xi \in {B}_{4}} \big| \partial^{\alpha}\widehat{g}(\xi)\big|$}, where the first maximum is taken over all multi-indexes $\alpha \in\Z_{\geq 0}^d$ of the form $\alpha = (\alpha_1, 0, 0 , \ldots, 0)$, with $0 \leq \alpha_1 \leq k(d)$. Then, for $s \in (0,2)$,
\begin{align}
\left|(K_g^\star)^{(j)}(s) \,s^{j-2}\right| & \leq 2^{j-2} \big(j^2 + 3j +6 + (j+2)\frak{r}_d\big) \,\sigma(\sph{d-1})^2 \,\mathcal{M}_d \ \ {\rm for} \ \ 2 \leq j \leq k(d) \label{20210713_16:19_6}
\end{align}
and
\begin{align}
\left|(K_g^\star)^{(1)}(s) \,s^{-1}\right| & \leq \frak{s}_d \,\sigma(\sph{d-1})^2 \,\mathcal{M}_d.  \label{20210713_16:19_1}
\end{align}
Here
\begin{align}\label{20210721_17:06}
\frak{s}_d := \max_{0 \leq s \leq 2}\left( \min\left\{ \frac{s^2 + 2s + 2 + (s + \tfrac12)\frak{r}_d}{s} \ , \ s^2 + 4s + 4 + (s + 2)\frak{r}_d\right\}\right).
\end{align}
\end{lemma}}
\noindent {{\sc Remark}: Note that\footnote{In principle, the values of $\frak{s}_d$ can be computed with arbitrary precision, since this amounts to solving a cubic equation on the variable $s$ in \eqref{20210721_17:06}. For simplicity, we shall use the stated bounds in the final computation.}
\begin{equation}\label{20210723_10:15}
7.88<\frak{s}_3 < 7.89 \ \ ; \ \ 7.67<\frak{s}_4 < 7.68\ \ ; \ \ 7.53<\frak{s}_5 < 7.54 \ \ ; \ \ 7.42< \frak{s}_6 < 7.43\ \ ; \ \ 7.34< \frak{s}_7 < 7.35.
\end{equation}}
\begin{proof}
{The idea is relatively simple, and matters boil down to certain standard computations, so we are brief with the details. We will take $j$ derivatives ($0 \leq j \leq k(d)$) of the expressions $\textup{I}(s)$, $\textup{II}(s)$ and $\textup{III}(s)$ in the form \eqref{eq:PrepareIII} (note that derivatives with respect to $s$ on $\widehat{g}$ will be associated to a multi-index $(\alpha_1,0,\ldots,0)$ in the way we set up things), use the triangle inequality and the elementary estimates
\begin{align}\label{20210721_13:53}
|s\,e_1 + \omega_3 + \omega_4| \leq 4; \ \ {\sup_{\xi \in {B}_{4}} \big| \partial^{\alpha}\widehat{g}(\xi)\big| \leq \mathcal{M}_d}.
\end{align}
Following this procedure, and using \eqref{20210720_23:08} when bounding the derivatives of $\textup{II}(s)$, we find,\footnote{The case $j=0$ will be used later on in the argument; see \S \ref{sec:Concl}.} for $0 \leq j \leq k(d)$, 
\begin{align}
\big|\textup{I}^{(j)}(s)\big| &\leq \big(j(j-1) + 2js + s^2\big) \,\sigma(\sph{d-1})^2 \,\mathcal{M}_d,\label{eq:Ijest}\\
\big| \textup{II}^{(j)}(s)\big|& \leq 2 \, \sigma(\sph{d-1})^2 \,\mathcal{M}_d.\label{eq:IIjest}
\end{align}
In the analysis of $\textup{III}(s)$ we can further take advantage of the fact that $\widehat{g}$ is non-negative on $\overline{B}_4$ as follows. Split the integral on the right-hand side of \eqref{eq:PrepareIII} into two integrals
\begin{equation}\label{eq:plusminus}
\textup{III}(s) = \textup{III}_+(s) + \textup{III}_{-}(s),
\end{equation}
where $\textup{III}_+(s)$ (resp. $\textup{III}_-(s)$) is the integral over the region where $(\omega_3 + \omega_4)_1 \geq 0$ (resp. $(\omega_3 + \omega_4)_1 < 0$). Then, for $0 < s < 2$, we have $\textup{III}_+(s)\geq 0$ and $\textup{III}_-(s)\leq0$. Moreover, using \eqref{20210721_13:53} and \eqref{20210720_23:09} we have
\begin{align*}
\big|\textup{III}_{+}(s)\big| & \leq s  \,\mc{M}_d \int_{\substack{(\sph{d-1})^2 \\ (\omega_3+\omega_4)_1 \geq0}}  (\omega_3+\omega_4)_1\, \d\sigma(\omega_3)\,\d\sigma(\omega_4) \\
& = \frac{s}{2}\,\mc{M}_d \int_{(\sph{d-1})^2}   \big|(\omega_3+\omega_4)_1\big| \, \d\sigma(\omega_3)\, \d\sigma(\omega_4) = \frac{s}{2}\,\frak{r}_d \, \sigma(\sph{d-1})^2\mc{M}_d,
\end{align*}
and the exact same bound holds for $\big|\textup{III}_{-}(s)\big|$. 
Going back to \eqref{eq:plusminus} and recalling that $\textup{III}_+(s)$ and $\textup{III}_-(s)$ have opposite signs, it follows that 
\begin{align}\label{20210723_11:08}
\big|\textup{III}(s)\big| \leq \frac{s}{2}\,\frak{r}_d \, \sigma(\sph{d-1})^2\mc{M}_d.
\end{align} 
Arguing similarly for the first derivative $\textup{III}^{(1)}(s)$ (for the part that retains $\widehat{g}$ we proceed as above, and for the part with $\partial^{\alpha}\widehat{g}$ we use \eqref{20210721_13:53} and \eqref{20210720_23:09}) we find
\begin{align}\label{20210721_16:01}
\big|\textup{III}^{(1)}(s)\big| \leq \left(s + \tfrac{1}{2}\right)\frak{r}_d \, \sigma(\sph{d-1})^2\mc{M}_d.
\end{align}
For $2 \leq j \leq k(d)$, only partial derivatives $\partial^{\alpha}\widehat{g}$ with $ |\alpha| \in \{j-1,j\}$ appear in $\textup{III}^{(j)}(s)$. One proceeds by applying the triangle inequality, \eqref{20210721_13:53} and \eqref{20210720_23:09} to get
\begin{align}
\big| \textup{III}^{(j)}(s)\big| & \leq (j + s )\,\frak{r}_d\, \sigma(\sph{d-1})^2 \,\mathcal{M}_d.\label{eq:IIIjest}
\end{align}
Adding up \eqref{eq:Ijest}, \eqref{eq:IIjest} and \eqref{eq:IIIjest} we find, for $2 \leq j \leq k(d)$:
\begin{align}\label{20210721_15:47}
\begin{split}
\left|(K_g^\star)^{(j)}(s)\right| &\leq \big|\textup{I}^{(j)}(s)\big|  + \big|\textup{II}^{(j)}(s)\big| + \big|\textup{III}^{(j)}(s)\big| \\
& \leq \left(\big(j(j-1) + 2js + s^2\big) + 2 + (j + s )\,\frak{r}_d\right)\sigma(\sph{d-1})^2 \,\mathcal{M}_d.
\end{split}
\end{align} 
Multiplying \eqref{20210721_15:47} by $s^{j-2}$ and using that $s \leq 2$ we arrive at \eqref{20210713_16:19_6}.
}

\smallskip

{Note that we already have two upper bounds for $\left|(K_g^\star)^{(1)}(s) \,s^{-1}\right|$. One comes from adding \eqref{eq:Ijest}, \eqref{eq:IIjest} (with $j=1)$ and \eqref{20210721_16:01}, and dividing by $s$, and the other one comes from \eqref{20210713_17:13}, in which we can use \eqref{20210721_15:47} (with $j=2$). We can take the minimum of these two upper bounds, i.e.
\begin{align*}
\left|(K_g^\star)^{(1)}(s) \,s^{-1}\right| & \leq \min\left\{ \frac{s^2 + 2s + 2 + (s + \tfrac12)\frak{r}_d}{s} \ , \ s^2 + 4s + 4 + (s + 2)\frak{r}_d\right\} \sigma(\sph{d-1})^2 \,\mathcal{M}_d\,,
\end{align*}
which leads to \eqref{20210713_16:19_1}.
}
\end{proof}

By the triangle inequality,  Lemmas \ref{Lem12_July13} and \ref{Lem13_July13} together yield the following bounds.

\begin{lemma}\label{Lem14_July14}
Let $d \in \{3,4,5,6,7\}$, $\frak{r}_d$ as in \eqref{20210720_23:09} and $\frak{s}_d$ as in \eqref{20210721_17:06}. Let ${\mc M}_d = \max_{\alpha} \sup_{\xi \in {B}_{4}} \big| \partial^{\alpha}\widehat{g}(\xi)\big|$, where the first maximum is taken over all multi-indexes $\alpha \in\Z_{\geq 0}^d$ of the form $\alpha = (\alpha_1, 0, 0 , \ldots, 0)$, with $0 \leq \alpha_1 \leq k(d)$. Then, for $t \in (-1,1)$, 
\begin{align*}
\big|K_g^{(j)}(t)\,(2 + 2t)^{j-1} \big| \leq {b_{d,j}} \, \sigma(\sph{d-1})^2 \,\mathcal{M}_d \ \ \ \ {(1 \leq j \leq k(d))}\,,
\end{align*}
with
\begin{align}
\begin{split}\label{20210723_10:16}
b_{d,1} & = \frak{s}_d \ \  \ ;\  \  \ b_{d,2}  = 16 + 4\frak{r}_d +\frak{s}_d \  \ \ ; \  \ \ b_{d,3}=96 + 22\frak{r}_d +3\frak{s}_d\ \ ; \\
 b_{d,4} &= 664 + 144 \frak{r}_d + 15 \frak{s}_d \  \ \ ;\ \  \ b_{d,5} = 5568 + 1166\frak{r}_d + 105\frak{s}_d.
\end{split}
\end{align}
\end{lemma}

\subsection{Final comparison} Recall that we are working under the specialization \eqref{20210713_12:36}, and the Gegenbauer expansion of $K_g$ is given by \eqref{20210528_14:54}. For $2\ell \geq k$, in light of \eqref{20210528_14:55} and Lemmas \ref{Lem_Wang_Sub}, \ref{Lem_Wang_Sub_2}, \ref{Lem14_July14}, we have
\begin{align}\label{20210714_10:33}
\big|\lambda_{d,g}(2\ell)\big| = \frac{2 \,\pi^{\nu +1}}{(2\ell + \nu) \, \Gamma(\nu)} \ \big|a_{2\ell}^{\nu}\big| \leq \left( \frac{2 \,\pi^{\nu +1}}{(2\ell + \nu) \, \Gamma(\nu)} \right) \left(\frac{D_{2\ell,k}^{\nu}  \left(C_{2\ell-k}^{\nu+k}(1)\right)^{1 - \frac{2}{p}} \big(\frak{h}_{2\ell-k}^{\nu+k} \big)^{\frac{2}{p}}}{(\frak{h}_{2\ell}^{\nu})^2 }\right) 2^{\frac{1}{p'}}\, {b_{d,k}} \, \sigma(\sph{d-1})^2 \,\mathcal{M}_d.
\end{align}
Writing the bounds on the right-hand sides of \eqref{20210601_09:57}--\eqref{20210601_09:61} as $- c_d\,\ell^{-d}$, from \eqref{eq:Compareg1} and \eqref{20210714_10:33} we seek
\begin{align}\label{20210714_10:43}
\left( \frac{2 \,\pi^{\nu +1}}{(2\ell + \nu) \, \Gamma(\nu)} \right) \left(\frac{D_{2\ell,k}^{\nu}  \left(C_{2\ell-k}^{\nu+k}(1)\right)^{1 - \frac{2}{p}} \big(\frak{h}_{2\ell-k}^{\nu+k} \big)^{\frac{2}{p}}}{(\frak{h}_{2\ell}^{\nu})^2 }\right) 2^{\frac{1}{p'}}\, {b_{d,k}} \, \sigma(\sph{d-1})^2 \,\mathcal{M}_d < c_d \, \ell^{-d}.
\end{align}
We now multiply both sides by $(c_d^{-1}2^{-d}) (2\ell)^d$ and  plug in the definitions of $D_{2\ell,k}^{\nu}$ in \eqref{20210714_10:38}, and $C_{2\ell-k}^{\nu+k}(1)$, $\frak{h}_{2\ell-k}^{\nu+k}$, $\frak{h}_{2\ell}^{\nu}$ in \eqref{eq:ValueAt1}. By isolating the terms that depend on $\ell$, we arrive at the following reformulation of \eqref{20210714_10:43}:
\begin{align}\label{20210715_17:39}
\beta_d \, G_d(\ell) \, \mathcal{M}_d < 1,
\end{align}
where 
\begin{align}\label{20210715_17:43}
\beta _d = \left(c_d^{-1}\,2^{-d} \right) \left(2^{\frac{1}{p'}}\, {b_{d,k}} \, \sigma(\sph{d-1})^2 \right) \! \left( \frac{2 \,\pi^{\nu +1}}{ \Gamma(\nu)} \right)\left(\prod_{j=0}^{k-1} 2(\nu +j) \!\right) \left(\frac{1}{\Gamma(2\nu + 2k)}\right)^{1 - \frac{2}{p}} \! \left( \frac{\pi \, 2^{1 - 2\nu - 2k} }{\Gamma(\nu + k)^2}\right)^{\frac{1}{p}} \! \left(\frac{\Gamma(\nu)^2}{\pi\, 2^{1 - 2\nu}}\right),
\end{align}
and, provided $2\ell\geq k$, we have that
\begin{align}
G_d(\ell) & := (2\ell)^d \left(\frac{1}{(2\ell + \nu)}\right)   \left(\prod_{j=0}^{k-1} \frac{1}{(2\ell -j)(2 \ell + 2\nu  + j)}\right)  \left( \frac{\Gamma(2\ell + 2\nu +k)}{(2\ell - k)!}\right)^{1 - \frac{2}{p}}  \nonumber \\
& \qquad \qquad   \qquad \qquad \qquad \qquad \qquad \times \left( \frac{\Gamma(2\ell + 2\nu + k)}{(2\ell - k)!\, (2\ell + \nu)}\right)^{\frac{1}{p}}\left(\frac{(2\ell)!\, (2\ell + \nu)}{\Gamma(2\ell + 2\nu)}\right) \label{20210714_11:48} \\
& = \frac{(2\ell)^{d}}{\left[(2\ell + \nu) \left(\prod_{j=0}^{k-1} (2\ell -j)(2 \ell + 2\nu  + j) \right) \left(\prod_{r = 1}^{2\nu -1}(2 \ell + r)\right)\right]^{\frac{1}{p}}}. \label{20210715_15:42}
\end{align}
In \eqref{20210714_11:48} we left clear where each term is coming from in \eqref{20210714_10:43}, and in \eqref{20210715_15:42} we proceeded with the full simplification, taking advantage of the fact that $2 \nu \in \N$ and the Gamma functions are all classical factorials. 

\subsubsection{The upper bound for $G_d$} With our specialization \eqref{20210713_12:36} one can check directly from \eqref{20210715_15:42} that $\lim_{\ell \to \infty} G_d(\ell) = 1$. Moreover, we actually have 
\begin{align}\label{20210715_15:48}
G_d(\ell) \leq 1
\end{align}
for any $\ell$ with $2\ell \geq k$. In order to establish \eqref{20210715_15:48}, we note that the terms in the first product in the denominator of \eqref{20210715_15:42} verify
\begin{align}\label{20210715_15:56}
(2\ell -1)(2\ell + 2\nu +1) \geq (2\ell -2)(2\ell + 2\nu +2) \geq  \ldots \geq  (2\ell -k +1)(2\ell + 2\nu +k-1) \,,
\end{align}
and the latter verifies
\begin{align}\label{20210715_15:57}
(2\ell -k +1)(2\ell + 2\nu +k-1)  \geq (2\ell)^2
\end{align}
provided $2\ell \geq (k-1)(2\nu + k-1) / (2\nu)$. This happens almost always, in which case \eqref{20210715_15:56} and \eqref{20210715_15:57} easily lead to \eqref{20210715_15:48}. The only cases left open (recall that we are assuming that $2\ell \geq k$) are $\ell=2$ in dimensions $d \in \{3,5,6\}$ and $\ell = 3$ in dimension $d=7$. In these cases, one simply checks directly that \eqref{20210715_15:48} holds.

\subsubsection{Conclusion}\label{sec:Concl} In light of \eqref{20210715_17:39} and \eqref{20210715_15:48} it suffices to have, for $2\ell \geq k$,  
\begin{equation}\label{20210715_18:58}
\mc{M}_d < \frac{1}{\beta_d}=: \mc{C}_d.
\end{equation}
{Now it is matter of carefully evaluating \eqref{20210715_17:43}. The bounds in \eqref{20210723_10:15} for $\frak{s}_d$, applied in the definition of $b_{k,d}$ in \eqref{20210723_10:16}, suffice to give us a 3-digit precision:}
\begin{align}\label{20210715_19:00}
{\mc{C}_3 = 0.157\ldots \ \ ; \ \ \mc{C}_4= 0.918\ldots \ \ ; \ \ \mc{C}_5 = 0.908\ldots \ \ ; \ \ \mc{C}_6 = 1.099\ldots\ \ ; \ \ \mc{C}_7 = 0.534\ldots\ .}
\end{align}

There is a final minor point left, which is to ensure the validity of \eqref{eq:Compareg1} for $2 \leq 2\ell < k$ (that is, $\ell = 1$ in dimensions $d \in \{3,4,5,6\}$, and $\ell\in\{1,2\}$ in dimension $d=7$). To see this, we start with the definition of the Gegenbauer coefficient $a_{2\ell}^{\nu}$ in \eqref{20210528_14:54} (recall also \eqref{eq:ValueAt1}), and apply the Cauchy--Schwarz inequality to obtain
\begin{align}\label{20210715_18:49}
\begin{split}
\big|a_{2\ell}^{\nu} \,(\frak{h}_{2\ell}^{\nu})^2 \big|  & =\left|\int_{-1}^1 K_g(t)\,C_{2\ell}^\nu(t) \,(1 - t^2)^{\nu - \frac12}\,\dt\right| \leq \frak{h}_{2\ell}^{\nu} \left(\int_{-1}^1 K_g(t)^2\,(1 - t^2)^{\nu - \frac12}\,\dt\right)^{\frac{1}{2}}\\
& \leq \frak{h}_{2\ell}^{\nu} \, \sup_{t \in (-1,1)}|K_g(t)|\, \left(\int_{-1}^1(1 - t^2)^{\nu - \frac12}\,\dt\right)^{\frac{1}{2}}\\
& =  \frak{h}_{2\ell}^{\nu}   \, \,\frak{h}_{0}^{\nu} \, \sup_{t \in (-1,1)}|K_g(t)|.
\end{split}
\end{align}
Invoking {\eqref{eq:Ijest}, \eqref{eq:IIjest} (when $j=0$) and \eqref{20210723_11:08}}, we find that 
\begin{equation}\label{20210715_18:50}
\sup_{t \in (-1,1)}|K_g(t)| {=} \sup_{s \in (0,2)}|K_g^\star(s)| \leq {(6 + \frak{r}_d)} \,\sigma(\sph{d-1})^2 \,\mathcal{M}_d.
\end{equation}
Then \eqref{20210528_14:55}, \eqref{20210715_18:49} and \eqref{20210715_18:50} plainly imply that 
\begin{align}\label{20210715_18:59}
\big|\lambda_{d,g}(2\ell) \big|= \frac{2 \,\pi^{\nu +1}}{(2\ell + \nu) \, \Gamma(\nu)} \ \big|a_{2\ell}^{\nu}\big| \leq \frac{2\,{(6 + \frak{r}_d)} \,\pi^{\nu +1}   \, \sigma(\sph{d-1})^2\, \frak{h}_{0}^{\nu}}{(2\ell + \nu) \, \Gamma(\nu)\,\frak{h}_{2\ell}^{\nu}}\, \mathcal{M}_d.
\end{align}
One can then directly check that condition \eqref{20210715_18:58}--\eqref{20210715_19:00} also implies that the right-hand side of \eqref{20210715_18:59} is strictly less than $c_d\ell^{-d}$ in the remaining cases $2 \leq 2\ell < k$. 

\smallskip

This concludes the proof of Lemma \ref{Lem6_crux} in the $C^k$-case, {and hence the proof of Theorem \ref{Thm1b}.}  

\section{Full classification of maximizers: proof of Theorem \ref{Thm2}}\label{sec:classification} Consider the subclass $\mc{X} \subset L^2(\sph{d-1})$ of characters given by
\begin{equation*}
\mc{X} = \big\{ f \in L^2(\sph{d-1})\ \ ; \ \ f(\omega) = c\,e^{iy \cdot  \omega}, \ {\rm for \ some}\  y \in \R^d \  {\rm and} \  c \in \C \setminus \{0\}\big\}.
\end{equation*}
\subsection{Part (i)} If $\widehat{g}= \bf{0}$ on the ball $\overline{B_4}$, for any $f \in L^2(\sph{d-1})$ we note that 
\begin{align*}
\int_{\R^d} \big|\widehat{f\sigma}(x)\big|^4\,g(x)\,\d x = \int_{(\sph{d-1})^4} \widehat g(\omega_1+\omega_2-\omega_3-\omega_4)  \prod_{j=1}^4 f_j(\omega_j) \, \d\sigma(\omega_j) = 0.
\end{align*}
Hence \eqref{20210717_00:33} and \eqref{eq:Q_h_decomp} yield
\begin{align*}
 \int_{\R^d} \big|\widehat{f\sigma}(x)\big|^4\,h(x)\,\d x = (2\pi)^d\,\big\|f\sigma* f\sigma\big\|_{L^2(\R^d)}^2 =  \int_{\R^d} \big|\widehat{f\sigma}(x)\big|^4\,\d x\,,
\end{align*}
and the complex-valued maximizers of \eqref{20210524_17:54} coincide with the complex-valued maximizers of the unperturbed adjoint Fourier restriction inequality, which have been classified in \cite{COS15, Fo15}. This is exactly the subclass $\mc{X}$.

\subsection{Part (ii)}\label{sec:Partii} The following chain of inequalities contains all the steps we followed in the previous sections in order to prove our sharp inequality:
    \begin{align}\label{eq:big_chain_ineqs}
    \begin{split}
            \left| \int_{\mathbb R^d} \big| \widehat{f\sigma}(x) \big|^4h(x)\, \d x\right|  &= \left| Q_h(f, f, \overline f, \overline f)\right| \\ 
            &\le Q_h(\lvert f\rvert, \lvert f\rvert, \lvert f\rvert, \lvert f\rvert) \qquad \ (\text{Lemma~\ref{Lem2_non-negative}})\\ 
            &\le Q_h(\lvert f\rvert_\sharp, \lvert f\rvert_\sharp, \lvert f\rvert_\sharp, \lvert f\rvert_\sharp) \ \ \ (\text{Lemma~\ref{Lem3_even}})\\ 
            &\le \tfrac14 H_{d, h}(\lvert f\rvert_\sharp^2) \qquad \qquad \ \ \ \, \ (\text{Section ~\ref{sec:cauchy_schwarz}})\\ 
            &\le  \frac{H_{d, h}(\boldsymbol{1})}{4\,\sigma(\mathbb S^{d-1})^2}\lVert f\rVert_{L^2(\mathbb S^{d-1})}^4 \,\,\,\,\,\, (\text{Prop.~\ref{Prop5}}).
        \end{split}
        \end{align}
This chain is sharp because all the inequalities above are equalities if $f$ is constant. So, if at least one of these inequalities is strict, then $f$ is not a maximizer.

\smallskip

We claim that, if $f \in L^2(\sph{d-1})$ is a maximizer, then $f \in \mc{X}$. In fact, by the conditions for equality in Lemma~\ref{Lem2_non-negative}, Lemma~\ref{Lem3_even} and Proposition~\ref{Prop5} (recall that in Lemma~\ref{Lem2_non-negative} we only stated a necessary condition) any maximizer  $f \in L^2(\sph{d-1})$ must verify   
     \begin{equation}\label{eq:necessary_conditions}
            \lVert f\sigma\ast f\sigma\rVert_{L^2(\mathbb R^d)}=\lVert \lvert f\rvert\sigma\ast \lvert f \rvert\sigma\rVert_{L^2(\mathbb R^d)} \ \ , \ \   \lvert f\rvert =\lvert f\rvert_\sharp \ \  \text{and } \ \ \lvert f\rvert_\sharp=\gamma\, \bf{1}\,, 
    \end{equation}
where $\gamma >0$ is a constant. These are exactly the same conditions that were used in \cite{COS15} in order to show that $f \in \mc{X}$. We briefly recall the argument. The first condition in \eqref{eq:necessary_conditions} implies that there exists a measurable function $\Phi: \overline{B_2} \to \C$ such that 
  \begin{equation}\label{eq:functional_equation}
        f(\omega_1)f(\omega_2)=\Phi(\omega_1+\omega_2)\lvert f (\omega_1) f(\omega_2)\rvert
    \end{equation}
for $\sigma^2$-a.e.\@ $(\omega_1, \omega_2) \in (\sph{d-1})^2$; see~\cite[Lemma~8]{COS15}. By the second and third conditions in \eqref{eq:necessary_conditions}, relation \eqref{eq:functional_equation} becomes
 \begin{equation*}
        f(\omega_1)f(\omega_2)=\gamma^2\, \Phi(\omega_1+\omega_2).   \end{equation*}
The only solutions to this functional equation are of the form $f(\omega)=c\, e^{\zeta\cdot \omega}$, for $\zeta\in \mathbb C^d$ and $c\in\mathbb C\setminus\{0\}$; see~\cite[Theorem~4]{COS15}. Finally, since $\lvert f \rvert$ is constant, we must have $\zeta=i y$ for some $y \in\mathbb R^d$, and $f \in \mc{X}$ as claimed.

\smallskip

This does not mean that {\it any} $f \in \mc{X}$ is a maximizer. In fact, we now verify that only the constant functions are maximizers in the general case. This is a distinct feature of our weighted setup.

\smallskip

Let $f(\omega)=c\,e^{iy\cdot \omega}$ for $y \in \R^d$ and $c \in \C\setminus\{0\}$. By our assumptions, since $\widehat{g}$ is continuous, non-negative, and not identically zero on $\overline{B_4}$, there is a subset $A\subset B_4$ of positive measure such that $\widehat{g}(\xi)>0$ for every $\xi\in A$. Then
\begin{equation}\label{eq:G_function}
    \frac{1}{|c|^4}\int_{\mathbb R^d} \lvert \widehat{f\sigma}(x)\rvert^4 \,h(x)\, \d x = \int_{\mathbb R^d} \lvert \widehat{\sigma}(x)\rvert^4\, \d x+\int_{\mathbb R^d} \lvert \widehat{\sigma}(x)\rvert^4 \,g(x-y)\, \d x.
\end{equation}
Define 
\begin{equation*}
\displaystyle G(y):=\int_{\mathbb R^d} \lvert \widehat{\sigma}(x)\rvert^4 \,g(x-y) \,\d x.
\end{equation*}
We claim that
\begin{equation}\label{eq:G_max_zero}
    \begin{array}{cc}
        |G(y)| <G(0)
    \end{array}
\end{equation}
for any $y \in \R^d \setminus\{0\}$. Using the triangle inequality in \eqref{eq:G_function}, followed by an application of \eqref{eq:G_max_zero}, we then conclude that the only maximizers to~\eqref{20210524_17:54} are the constant functions. To prove our claim, we start by observing that {$G \in L^2(\R^d)$}, and that 
\begin{equation}\label{eq:Ghat_Positive}
        {\displaystyle\widehat{G}(z)= (2\pi)^d\,\widehat{g}(z)\, (\sigma\ast\sigma\ast \sigma\ast \sigma)(z).}
\end{equation}
Hence $\widehat{G}$ is supported on $\overline{B_4}$, and is radial and non-negative. Moreover, $\widehat{G}(z) >0$ on $A$ since $(\sigma\ast\sigma\ast \sigma\ast \sigma)(z)>0$ on $B_4$. By Fourier inversion, for $y\in\mathbb R^d\setminus\{0\}$, we then have
\begin{align*}
G(0) \pm G(y) & = (2\pi)^{-d} \int_{\overline{B_4}} \widehat{G}(z) \left(1 \pm e^{-iy\cdot z}\right)\,\d z = (2\pi)^{-d} \int_{\overline{B_4}} \widehat{G}(z) \,\big(1 \pm \cos(y\cdot z)\big)\,\d z\\
& \geq (2\pi)^{-d} \int_{A} \widehat{G}(z) \,\big(1 \pm \cos(y\cdot z)\big)\,\d z >0, 
\end{align*}
since $1 \pm \cos(y\cdot z) >0$ for a.e. $z \in \R^d$. The claim~\eqref{eq:G_max_zero} readily follows, and the proof of Theorem \ref{Thm2} is now complete.

\section{Some related examples} \label{Sec_9}
In this section we highlight a few examples that motivate our choice of assumptions: the radiality condition in (R1), the non-negativity condition in (R1), and the smallness condition in \eqref{20210525_09:45} and \eqref{20210710_19:14} (associated to (R2.A) and (R2.C), respectively). This is to be understood in the following sense: if one of these conditions is removed entirely (while keeping the other ones) it is possible to construct examples of perturbations $g$ for which the constant functions fail to maximize  \eqref{20210524_17:54}.
In the last subsection, we further address the particular case of quadratic weights; this leads to an improved result in dimensions $d\in\{3,4,5,6,7\}$, and to a new sharp inequality when $d=8$ {(Theorem \ref{thm:NewSharpIneq} above).}

\subsection{Critical points and radiality}
In light of inequality \eqref{20210524_17:54}, it is natural to consider the functional
\begin{equation}\label{eq:PhiFunctional}
    \Phi(f)=\frac{\lvert \int_{\R^d} |\widehat{f\sigma}(x)|^4 h(x) \d x\rvert }{ \|f\|_{L^2(\sph{d-1})}^4}.
\end{equation}
Here $h={\bf 1}+g$, and we assume for the next proposition that $g$ satisfies (R1), with ``radial" replaced merely by ``even", and (R4.A) or (R4.C). The constant function $\mathbf{1}$ is a maximizer of~\eqref{20210709_22:16} if and only if $\Phi(f)\le \Phi(\mathbf 1)$ for all $f\in L^2(\mathbb S^{d-1})$. 
A necessary condition for this is that ${\bf 1}$ be a critical point of the functional $\Phi$, and we provide a characterization of this condition in the following result. Note that, by the non-negativity of $\widehat{g}\big|_{\overline{B_4}}$, the numerator of~\eqref{eq:PhiFunctional} is strictly positive when $f=\mathbf{1} $ (recall~\eqref{20210526_15:39}), hence $\Phi$ is differentiable at $\mathbf{1}$. 

\begin{proposition}[Motivation for the radiality assumption]\label{prop:criticalpoints}
The function ${\bf 1}$ is a critical point of $\Phi$ if and only if $\widehat g\big|_{\overline{B_4}}\ast \sigma^{\ast 3}$ is constant on $\sph{d-1}$.
\end{proposition}

\noindent As an immediate consequence, there exist non-radial even functions $\widehat{g}:\overline{B_4} \to\R_{\geq 0}$ for which ${\bf 1}$ fails to maximize the corresponding inequality \eqref{20210524_17:54} (but see the remark after the proof of Proposition \ref{prop:criticalpoints}). 
A simple example is obtained by setting $\xi_\pm:=(\pm 4,0,\ldots,0)\in\R^d$, and considering a non-zero, even function $\widehat g:\R^d\to\R_{\geq 0}$ which is supported on the closure of $B(\xi_-,1)\cup B(\xi_+,1)$ and strictly positive on $B(\xi_-,1/2)\cup B(\xi_+,1/2)$. In this case, $(\widehat g\big|_{\overline{B_4}}\ast \sigma^{\ast 3})(\pm1,0,\ldots,0)>0$, but since the support of the latter function does not contain $\sph{d-1}$, it cannot be constant there, and by Proposition \ref{prop:criticalpoints} the function {\bf 1} is not a critical point of $\Phi$.  
\begin{proof}[Proof of Proposition \ref{prop:criticalpoints}]
As in \cite[\S 16]{ChSh1}, we may restrict attention to functions of the form $f={\bf 1}+\lambda\varphi$, where $\varphi\perp{\bf 1}$, $\varphi$ is real-valued and even, $\lambda>0$ is small enough, and $\|\varphi\|_{L^2(\sph{d-1})}=1$.
A straightforward calculation yields
\[\frac{\d}{\d\lambda}\bigg\rvert_{\lambda=0} \Phi({\bf 1}+\lambda\varphi)
=4\|{\bf 1}\|_{L^2(\sph{d-1})}^{-4} \int_{\R^d} \widehat{\varphi\sigma}(x)\,\widehat{\sigma}^3(x)\,h(x)\,\d x 
\simeq_d \int_{\R^d} (\varphi\sigma\ast\sigma^{\ast 3})(\xi) \,\widehat{h}(\xi) \,\d\xi,\] 
where the last identity follows from Plancherel's theorem. That $\widehat{h}$ can be replaced by $\widehat{g}$ in the latter integral follows from \eqref{eq:hHat} together with the observation that\footnote{By a slight abuse of notation, we are using the fact that the spherical convolution defines a radial function on $\R^d$ and, given $r\geq 0$, denote by $\sigma^{\ast 3}(r)$ the value attained on the sphere $|\xi|=r$.}
\[\int_{\R^d} (\varphi\sigma\ast\sigma^{\ast 3})(\xi) \,\boldsymbol{\delta}(\xi) \,\d\xi = (\varphi\sigma\ast\sigma^{\ast 3})(0)=\sigma^{\ast 3}(1)\int_{\sph{d-1}}\varphi\,\d\sigma=0.\]
Consequently,
\begin{align}
\frac{\d}{\d\lambda}\bigg\rvert_{\lambda=0} \Phi({\bf 1}+\lambda\varphi)
&\simeq_d \int_{\R^d} (\varphi\sigma\ast\sigma^{\ast 3})(\xi) \,\widehat{g}(\xi)\, \d\xi
=\int_{B_4} \int_{(\sph{d-1})^4} \varphi(\omega_1)\ddirac{\xi-\sum_{j=1}^4\omega_j} \d\sigma(\vec\omega) \,\widehat{g}(\xi) \,\d\xi\notag\\
&=\int_{(\sph{d-1})^4} \varphi(\omega_1)\,\widehat{g}\left(\sum_{j=1}^4\omega_j\right)\d\sigma(\vec\omega) =\int_{\sph{d-1}}\varphi\gamma\,\d\sigma,\label{eq:phigamma}
\end{align}
where $\gamma(\omega):=\int_{(\sph{d-1})^3}\widehat{g}\left(\omega+\sum_{j=2}^4\omega_j\right)\d\sigma(\omega_2)\d\sigma(\omega_3)\d\sigma(\omega_4)$.
One easily checks that $\gamma=(\widehat g\big|_{\overline{B_4}}\ast\sigma^{\ast 3})\big|_{\sph{d-1}}$.
From \eqref{eq:phigamma}, it then follows that ${\bf 1}$ is a critical point of $\Phi$ if $\gamma$ is constant, which establishes the first assertion of the proposition.
For the converse direction, start by noting that $\gamma$ defines a continuous, even, non-negative function on $\sph{d-1}$ under our assumptions.
If $\gamma$ is not constant, then there exist $\omega_1\neq\pm\omega_2\in\sph{d-1}$ such that $\gamma(\omega_1)>\gamma(\omega_2)\geq 0$.
Let $\delta>0$ be small enough, such that
\begin{equation}\label{eq:choosedelta}
\inf_{\omega\in\mathcal C(\omega_1,\delta)} \gamma(\omega)>\sup_{\omega\in\mathcal C(\omega_2,\delta)} \gamma(\omega), \text{ and so } \mathcal C(\omega_1,\delta)\cap \mathcal C(\omega_2,\delta)=\emptyset,
\end{equation}
where $\mathcal C(\omega,r)\subset\sph{d-1}$ denotes the open cap of radius $r$ centered at $\omega$.
Let $\varphi={\bf 1}_{\mathcal C(\omega_1,\delta)}+{\bf 1}_{\mathcal C(-\omega_1,\delta)}-{\bf 1}_{\mathcal C(\omega_2,\delta)}-{\bf 1}_{\mathcal C(-\omega_2,\delta)}$.
Then $\varphi\in L^\infty(\sph{d-1})$ is non-zero, real-valued and even, and such that $\varphi\perp{\bf 1}$. 
Moreover, \eqref{eq:choosedelta} forces 
\[\frac12\int_{\sph{d-1}}\varphi\gamma\, \d\sigma= \int_{\mathcal C(\omega_1,\delta)} \gamma\,\d\sigma-\int_{\mathcal C(\omega_2,\delta)} \gamma\,\d\sigma>0,\]
which in light of \eqref{eq:phigamma} implies that ${\bf 1}$ is not a critical point of $\Phi$.
This concludes the proof.
\end{proof}

\noindent {\sc Remark:} As noted after the statement of Proposition \ref{prop:criticalpoints}, a non-radial function $G\colon \overline{B_4}\to \mathbb R$ will in general fail to satisfy the property that $G\ast \sigma^{\ast 3}$ is constant on $\sph{d-1}$. However, and perhaps surprisingly, there exist non-radial functions $G$ for which this property does hold. A simple example in dimension $d=3$ is given by
\begin{equation}\label{eq:non_radial_example}
    \begin{array}{ccc}
    \displaystyle    G(y)=u(\lvert y \rvert)\frac{y}{\lvert y \rvert} \cdot e_3, \text{ where } u(r)=r(2-r)\left(r-\frac{35}{26}\right)\mathbf 1_{[0,2]}(r),& y\in\mathbb R^3, r>0;
    \end{array}
\end{equation}
here, $e_3=(0,0,1)$. To prove this, recall that $\sigma^{\ast 3}$ is supported on $\overline{B_3}\subset\R^3$ and satisfies $\sigma^{\ast 3}(\lvert y \rvert)= 8\pi^2$ for $\lvert y\rvert \le 1$ and $\sigma^{\ast 3}(\lvert y \rvert)=4\pi^2(3/\lvert y \rvert-1)$ for $1\leq\lvert y \rvert \le 3$; see~\cite[Eq.\@ (3.11)]{OSQu2}. Given $\omega\in\mathbb S^2$, let $R_\omega\in \operatorname*{SO}(3)$ be a rotation such that $R_\omega \omega=e_3$. Hence, 
\begin{equation}\label{eq:GastSigmaThree}
        (4\pi)^{-2}(G\ast \sigma^{\ast 3})(\omega)=
        \int_{e_3+B_1} 2G(R_\omega^{-1} y)\, \d y 
        + 
        \left( \int_{e_3+B_3}-\int_{e_3+B_1}\right) 
        G(R_\omega^{-1} y)\left(-1 +\frac{3}{\lvert y - e_3 \rvert}\right)\, \d y.
\end{equation}
Letting $\tilde{\omega}=R_\omega e_3$ and $r=\lvert y \rvert$, we have that $G(R_\omega^{-1} y)=\frac{u(r)}{r} \,y_3\, \tilde{\omega}_3 + \frac{u(r)}{r}(y_1, y_2, 0)\cdot \tilde \omega$. Observe that the second summand does not contribute to any of the integrals in~\eqref{eq:GastSigmaThree}, as the change of variables $(y_1,y_2, y_3)\mapsto (-y_1, -y_2, y_3)$ reveals. Letting $t=y_3/r$ and invoking the fact that $u$ is supported on $\overline{B_2}\subset\R^3$, the right-hand side of \eqref{eq:GastSigmaThree}  then reads, up to an irrelevant factor of $6\pi \tilde{\omega}_3$,\footnote{Here we are using the facts that $|y-e_3|\leq 1$ if and only if $t\geq \frac r2$, and that $|y-e_3|\leq 3$ if and only if $t\geq \frac{r^2-8}{2r}$.}
\begin{equation*}\label{eq:GastSigmaThreePolar}
\begin{split}
&\int_0^4 u(r)r^2\left[ \int_{\min\{r/2, 1\}}^1 t\left( 1-\frac{1}{(r^2-2rt+1)^\frac12}\right)\, \d t  + \int_{\max\{\frac{r^2-8}{2r}, -1\}}^1 t\left( -\frac13 +\frac{1}{(r^2-2rt+1)^\frac12}\right)\, \d t\right]\d r= \\
&\int_0^2 u(r)r^2\left[ \int_{r/2}^1 t\left( 1-\frac{1}{(r^2-2rt+1)^\frac12}\right)\, \d t  + \int_{-1}^1 t\left( -\frac13 +\frac{1}{(r^2-2rt+1)^\frac12}\right)\, \d t\right]\d r=\int_0^2 u(r)\frac{r^3(8-3r)}{24} \, \d r, \\ 
\end{split}
\end{equation*}
and the latter integral vanishes, since the function $u$ was defined precisely to ensure this (as a side remark, note that any non-zero function $u=u(r)$ supported on $[0,2]$ and orthogonal to the quartic polynomial $r^3(8-3r)$ on that interval would work). We conclude that $G\ast \sigma^{\ast 3} \equiv 0$ on $\mathbb S^2$, even though $G$ is plainly non-radial on $\overline{B_2}\subset\R^3$.


\subsection{Non-negativity and smallness}
In this section, we construct examples revealing that the non-negativity condition in (R1) and the smallness condition in \eqref{20210525_09:45} and \eqref{20210710_19:14} (associated to (R2.A) and (R2.C), respectively) cannot be dropped entirely from the set of running assumptions.  As usual, $J_{\alpha}$ will denote the Bessel function of the first kind of order $\alpha$. Considering a radial weight $h$, we will require the following computations, for $\varphi\in L^2(\mathbb S^{d-1})$ satisfying $\varphi\perp{\bf 1}$: 
\begin{equation*}
        \begin{split}
          \left.\frac{\d^2}{\d\lambda^2}\right|_{\lambda=0}\left( \int_{\mathbb S^{d-1}} \lvert \mathbf 1 + \lambda \varphi\rvert^2\, \d\sigma\right)^2 &= 4\,\sigma(\mathbb S^{d-1}) \int_{\mathbb S^{d-1}} \lvert \varphi\rvert^2\, \d\sigma, \\
            \left.\frac{\d^2}{\d\lambda^2}\right|_{\lambda=0} \int_{\mathbb R^d} \lvert \widehat\sigma+\lambda\widehat{\varphi\sigma}\rvert^4 h\, \d x &= 8\int_{\mathbb R^d} \widehat{\sigma}^2\lvert \widehat{\varphi \sigma}\rvert^2 h\, \d x + 4 \int_{\mathbb R^d} \widehat{\sigma}^2\Re[(\widehat{\varphi\sigma})^2]\, h\,\d x.
            \end{split}
        \end{equation*}   
By the previous subsection, $\mathbf 1$ is a critical point of the functional $\Phi$. We compute its second variation:
\begin{equation}\label{eq:PhiSecondDerivative}
    \left.\frac{\d^2}{\d\lambda^2}\right|_{\lambda=0}\Phi(\mathbf 1 + \lambda \varphi)\simeq_d \int_{\mathbb R^d} \widehat{\sigma}^2\left( 2\lvert \widehat{\varphi\sigma}\rvert^2+ \Re[(\widehat{\varphi\sigma})^2]\right)h\, \d x - \frac{\lVert \varphi \rVert_{L^2(\mathbb S^{d-1})}^2}{\sigma(\mathbb S^{d-1})}\int_{\mathbb R^d} \widehat{\sigma}^4 h\, \d x .
\end{equation}
A necessary condition for constant functions to maximize~\eqref{20210524_17:54} is that such second variation be non-positive for all test functions $\varphi$. We will consider $\varphi\in\{Y_2, iY_1\}$, where $Y_k$ denotes a real spherical harmonic on $\mathbb S^{d-1}$ of degree $k$. We recall that $\nu=d/2-1$ and the formula from \cite[Eq.\@ (2.3)]{COSS2}:
       \begin{equation}\label{eq:FTSphHarm}
                 \widehat{Y_k\sigma}(x)=\frac{(2\pi)^\frac{d}2}{i^k}\frac{J_{\nu+k}(\lvert x\rvert)}{\lvert x \rvert^\nu}Y_k\left(\frac{x}{\lvert x \rvert}\right).   
         \end{equation}
\begin{proposition}[Motivation for the smallness assumption]\label{prop:R4Nec}
Assume $d=3$. There exist $c_0>0$ and a radial Schwartz function $g:\R^3 \to \R$ with $\widehat{g}\ge 0$, such that, letting $h=\mathbf{1}+ c g$ for $c>0$, the following holds: 
\begin{equation}\label{eq:RFourNecessary}
    \begin{array}{cc}
    \displaystyle \left.\frac{\d^2}{\d\lambda^2}\right|_{\lambda=0}\Phi(\mathbf 1 + \lambda Y_2) \le 0& \text{if and only if } c\leq c_0.
    \end{array}
\end{equation}
In particular, the function {\bf 1} fails to maximize~\eqref{20210524_17:54} if $c>c_0$.
\end{proposition}
\begin{proof}
Specializing~\eqref{eq:PhiSecondDerivative} to $\varphi=Y_2$ and invoking~\eqref{eq:FTSphHarm}, we obtain via Plancherel's identity that
\begin{equation}\label{eq:SecondVarYTwo}
    \left.\frac{\d^2}{\d\lambda^2}\right|_{\lambda=0}\Phi(\mathbf 1 + \lambda Y_2)\simeq_d \int_0^\infty  H_2(r)h(r)\, r^{d-1}\, \d r\simeq_d\int_0^\infty \widehat{H}_2(\rho) \widehat{h}(\rho)\, \rho^{d-1}\, \d \rho,
\end{equation}
where $H_2(r):=J_\nu^2(r)[3J_{\nu+2}^2(r)-J_\nu^2(r)]r^{4-2d}$. 
On the right-hand side of \eqref{eq:SecondVarYTwo}, the hats refer to the Fourier transform of $h$ and $H_2$ seen as functions of the radial variable of $\mathbb R^d$, i.e.,
\begin{equation}\label{eq:RadialFT}
    \widehat{h}(\lvert\xi\rvert)=\int_{\mathbb R^d} h(\lvert x\rvert) e^{-ix\cdot \xi}\, \d x = (2\pi)^\frac{d}{2} \int_0^\infty h(r)\frac{J_\nu(r\lvert\xi\rvert)}{r^\nu \lvert\xi\rvert^\nu} r^{d-1}\, \d r,
\end{equation}
and similarly for $\widehat{H}_2$. Henceforth we assume $d=3$, whence $\nu=1/2$; one checks that $\widehat{H}_2$ is supported on the interval $[0,4]$ and given by
\begin{equation*}
    \widehat{H}_2(\rho)=
    \begin{cases}
      -{\frac {9}{80}}\,{\rho}^{5}+{\frac {9}{20}}\,{\rho}^{4}-{\frac {9}{8}}\,{\rho}^{3}+3\,{\rho}^{2}-3\,\rho-\frac{8}{5}
, & \rho\in[0, 2], \\ 
     \frac{1}{80}(\rho-4)^2(3\rho^4 - 12\rho^3 + 6\rho^2 -16)\rho^{-1}
, & \rho\in [2, 4].
    \end{cases}
\end{equation*}
Next we observe that the polynomial $3\rho^4 - 12\rho^3 + 6\rho^2 -16$ has exactly one positive root $\rho_0=3.556\ldots$, and that $\widehat{H}_2(\rho)<0$ for $\rho\in [0, \rho_0)$, whereas $\widehat{H}_2(\rho)>0$ for $\rho\in(\rho_0, 4)$; see Figure~\ref{fig:HTwoGraph}. 
\begin{figure}
    \centering
    \includegraphics[width=0.3\textwidth]{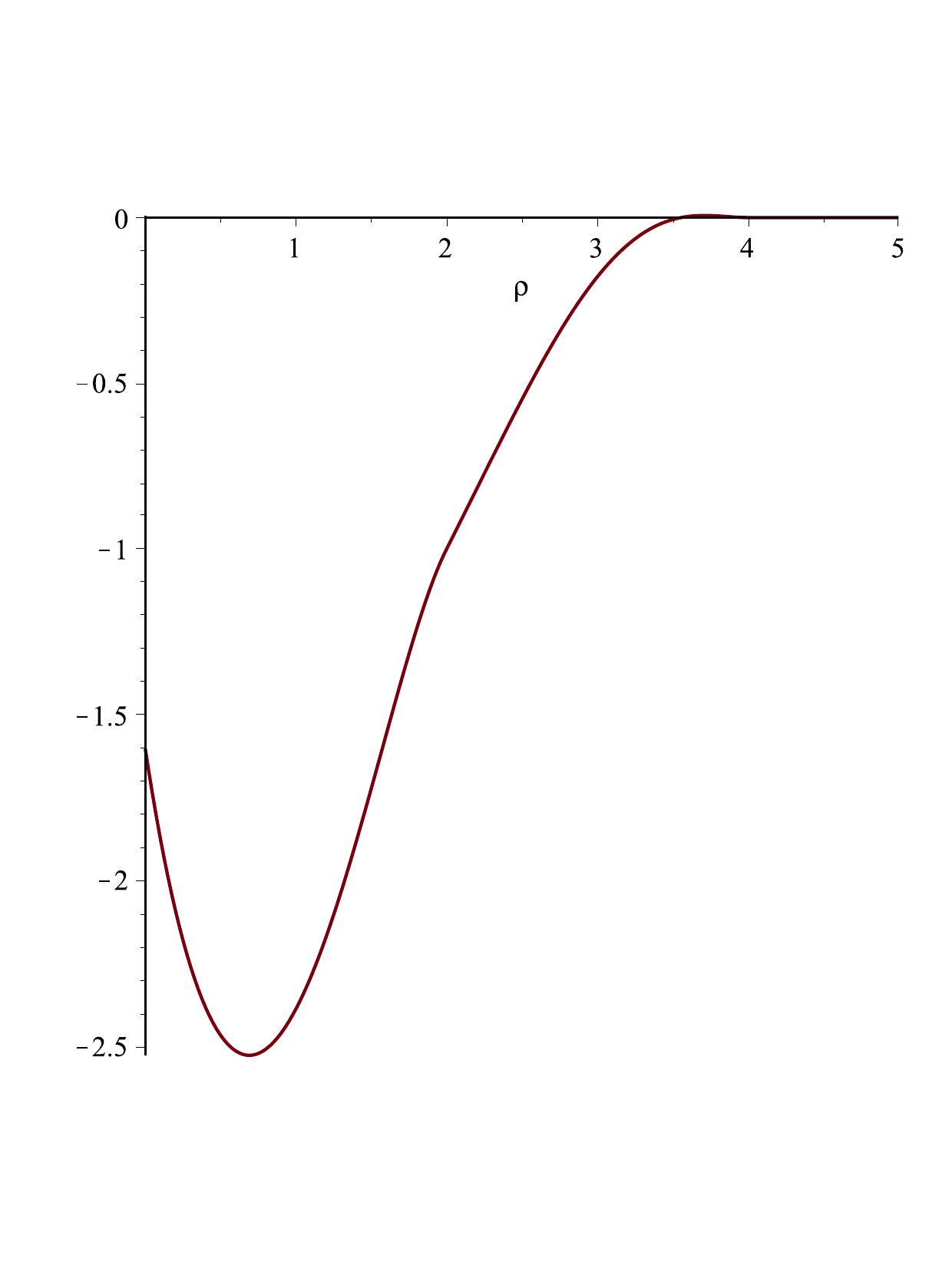}
    \includegraphics[width=0.3\textwidth]{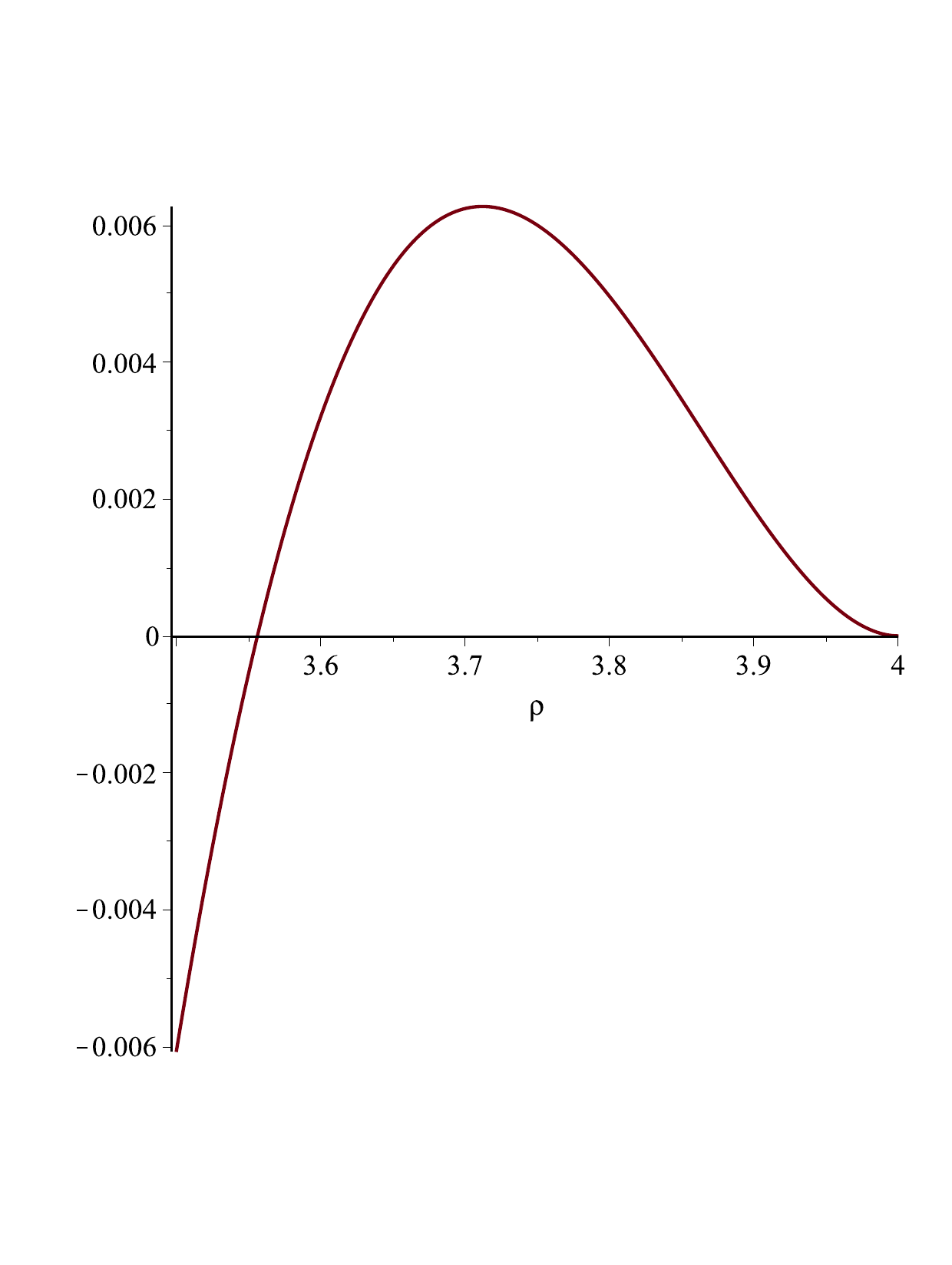}
    \caption{Graphs of $\widehat{H}_2=\widehat{H}_2(\rho)$ for $\rho\in [0, 5]$ and $\rho\in[\frac72, 4]$, respectively.}
    \label{fig:HTwoGraph}
\end{figure}
Consider a non-zero, smooth, compactly supported bump function $\psi\in C_0^\infty(\R)$ such that $\psi\geq 0$ and supp$(\psi)\subseteq[\rho_0,4]$, and define the radial function $g:\R^3\to\R$ via $\widehat g=\widehat g(|\cdot|)=\psi$ (hence $g$ is Schwartz). Recalling that $h(\lvert x \rvert)=1 + c \,g(\lvert x\rvert)$, hence $\widehat{h}(\lvert \xi\rvert)=(2\pi)^3\boldsymbol{\delta}(\xi) + c\, \widehat{g}(\lvert \xi\rvert)$, we infer from~\eqref{eq:SecondVarYTwo} that 
\begin{equation*}
    \left.\frac{\d^2}{\d\lambda^2}\right|_{\lambda=0}\Phi(\mathbf 1 + \lambda Y_2)\simeq_d -\frac{8}{5}(2\pi)^3+ c \int_{\rho_0}^4 \widehat{H}_2(\rho) \psi(\rho)\rho^2\, 
    \d\rho, 
\end{equation*}
for which \eqref{eq:RFourNecessary} holds with $c_0=\frac85(2\pi)^3 \left(\int_{\rho_0}^4 \widehat{H}_2(\rho) \psi(\rho)\rho^2 \d\rho\right)^{-1}>0$. 
This concludes the proof.
\end{proof}
\begin{proposition}[Motivation for the non-negativity assumption] Assume $d=3$. 
There exist $c_0>0$ and a radial and smooth $g\in L^1(\R^3)\cap L^\infty(\mathbb R^3)$ satisfying $\widehat{g}(\lvert \xi\rvert)\le 0$ for all $\xi\in \overline{B_4}$, with the following property.  If $c\in(0,c_0)$ and $h=\mathbf{1}+c g$, then
\begin{equation}\label{eq:SecondVariYone}
    \left.\frac{\d^2}{\d\lambda^2}\right|_{\lambda=0}\Phi(\mathbf 1 + \lambda iY_1) >0.
\end{equation}
In particular, the function {\bf 1} fails to maximize~\eqref{20210524_17:54} if $c<c_0$. 
\end{proposition}
\begin{proof}
    In a similar way to the proof of Proposition \ref{prop:R4Nec}, identities \eqref{eq:PhiSecondDerivative} and~\eqref{eq:FTSphHarm} yield
    \begin{equation}\label{eq:SecondVarYOneExpanded}
        \left.\frac{\d^2}{\d\lambda^2}\right|_{\lambda=0} \Phi(\mathbf 1 + \lambda iY_1)\simeq_d \int_0^\infty H_1(r) r^{d-1}\, \d r + c \int_0^\infty H_1(r) g(r) r^{d-1}\, \d r,
    \end{equation}
provided $c\in (0,c_0)$ and $c_0>0$ is sufficiently small (to be chosen below), where $H_1(r)=J_\nu^2(r)[ 3J_{\nu+1}^2(r)-J_\nu^2(r)] r^{4-2d}$. The first integral on the right-hand side of \eqref{eq:SecondVarYOneExpanded} vanishes (see~\cite[§5.1]{GoN}); we remark that this is a direct consequence of the modulation invariance of $\Phi$ in the unweighted case $h={\bf 1}$. For $d=3$, we compute $\widehat{H}_1$ explicitly; it is supported on the interval $[0,4]$ and given by
    \begin{equation}\label{eq:HOneFT}
        \widehat{H}_1(\rho)=\begin{cases} \frac18\rho^2(3\rho - 8), & \rho\in [0, 2], \\
        -\frac18 \rho(\rho-4)^2, & \rho\in [2, 4].
        \end{cases}
    \end{equation}
    Let $g=g(|\cdot|):=H_1$, which is smooth and belongs to $L^1(\R^3)\cap L^\infty(\mathbb R^3)$, and satisfies $\widehat{g}\le 0$ everywhere, as can easily be checked from~\eqref{eq:HOneFT}. 
    Next we choose $c_0>0$ small enough so that the numerator in \eqref{eq:PhiFunctional} is non-negative, thus ensuring differentiability of $\Phi$.
    Then~\eqref{eq:SecondVarYOneExpanded} immediately implies that $\left.\frac{\d^2}{\d\lambda^2}\right|_{\lambda=0} \Phi(\mathbf 1 + \lambda iY_1)>0$, and the proof is complete.
\end{proof}

\subsection{Quadratic weights}\label{quadweights}
In this subsection, we observe that all computations can be carried out explicitly if the perturbation is of the form $\widehat g(\xi)=(a+b|\xi|^2){\bf 1}_{\overline{B_4}}(\xi)$. This leads to an improved result in the dimension range $d\in\{3,4, 5, 6, 7\}$, as well as to a new sharp inequality in dimension $d=8.$\footnote{Our methods work to generate similar sharp inequalities in dimensions $d>8$ as well, but we keep the discussion only with $d=8$ for simplicity, as a proof of concept.} The next result 
should be compared to the remark immediately following the statement of Theorem \ref{Thm1}.

\begin{proposition}\label{prop:tildeLd}
Let $g \in L^2(\R^d) \cap L^{\infty}(\R^d)$ be such that $\widehat{g}(\xi)=|\xi|^2$ if $|\xi|\leq 4$, and set $h={\bf 1}+b g$, for some $b> 0$.
For each $3\leq d\leq 7$, 
constant functions are the unique complex-valued maximizers of \eqref{20210524_17:54}. 
\end{proposition}


\noindent {\sc Remark:} If $d=3$ and $\widehat g=|\cdot|^2 \,{\bf 1}_{\overline{B_4}}$, then setting $b=1$ and invoking \eqref{eq:FTSphHarm} with $k=0$, we have that
\begin{equation}\label{eq:hinthiscase}
h(|x|)=1+\frac{1}{(2\pi)^3}\int_{|\xi|\leq 4} |\xi|^2 e^{ix\cdot\xi} \d\xi=1-\frac{\cos(4|x|)}{\pi^2 |x|^2}\left(32-\frac{12}{|x|^{2}}\right)+\frac{\sin(4|x|)}{\pi^2 |x|^3}\left(24-\frac{3}{|x|^2}\right),
\end{equation}
which is a {\it signed} radial weight for which constants are maximizers. In fact, one easily checks that 
$h(\pi)=1-\frac{32}{\pi^4}+\frac{12}{\pi^6}>0$ but 
$h(\frac{\pi}2)=1-\frac{128}{\pi^4}+\frac{192}{\pi^6}<0$. 
\begin{proof}[Proof of Proposition \ref{prop:tildeLd}]
The following {modified magical identity follows immediately from Proposition~\ref{prop:magic_id} applied to $h = {\bf 1}$, and holds for all non-negative  even $f\in L^2(\mathbb S^{d-1})$:}
\begin{equation}\label{eq:mod_magic_id}
    \int_{\mathbb R^d}\big(\widehat{f\sigma}(x)\big)^4(1+bg(x))\,\d x= \frac34 \int_{\mathbb R^d} \lvert \nabla \big((\widehat{f\sigma})^2\big)(x)\rvert^2\,\d x+b\int_{\mathbb R^d}\big(\widehat{f\sigma}(x)\big)^4g(x)\,\d x.
\end{equation}
Applying the Cauchy--Schwarz inequality like in §\ref{sec:cauchy_schwarz} we obtain the following modified version of~\eqref{20210527_14:23}:
\begin{equation}\label{eq:mod_cauchy_schwarz}
    \int_{\mathbb R^d}\big(\widehat{f\sigma}(x)\big)^4(1+bg(x))\,\d x \le \frac14\int_{(\sph{d-1})^2}\, f(\omega_1)^2f(\omega_2)^2\, \widetilde{M}_b(\omega_1, \omega_2)\, \d \sigma(\omega_1)\, \d \sigma(\omega_2), 
\end{equation}
where the kernel function $\widetilde{M}_b$ is now given by 
\begin{equation}\label{eq:mod_kernel}
    \begin{split}
        \widetilde{M}_b(\omega_1, \omega_2)&= K_{\bf 1}(\omega_1\cdot \omega_2)+4b \int_{(\mathbb S^{d-1})^2}\widehat{g}\left( \sum_{j=1}^4 \omega_j\right)\, \d \sigma(\omega_3)\,\d\sigma(\omega_4) \\ 
        &= K_{\bf 1}(\omega_1\cdot \omega_2) + 8\,b\,\sigma(\mathbb S^{d-1})^2\, (2+\omega_1\cdot\omega_2),
    \end{split}
\end{equation}
with $K_{\bf 1}$ given by~\eqref{eq:kappad}, and we used that $\widehat{g}(\xi)=\lvert \xi\rvert^2$ for $\lvert\xi\rvert\le 4$ to compute the integral. Only the first summand in~\eqref{eq:mod_kernel} gives a nontrivial contribution to~\eqref{eq:mod_cauchy_schwarz}, since the function $f$ that appears there is even. We can therefore continue the proof like in the unperturbed case $b=0$, in which case the constant functions are maximizers~(\cite{COS15}). The argument of Section \ref{sec:classification} shows that constants are the unique maximizers, since $b\ne 0$.  
\end{proof}
\noindent {\sc Remark:} The condition $b>0$ in Proposition \ref{prop:tildeLd} plays the same role as the non-negativity requirement in the assumption (R1), and the proof breaks down at several places if such condition is dropped. For instance, it is unclear whether~\eqref{eq:mod_cauchy_schwarz} would hold if $b<0$. 
\smallskip

The modified magical identity~\eqref{eq:mod_magic_id} worked well in the proof of Proposition \ref{prop:tildeLd} (leading to simpler computations) since the unperturbed problem had already been solved in~\cite{COS15}. However, the full magical identity from Proposition~\ref{prop:magic_id}  is needed in the proof of Theorem \ref{thm:NewSharpIneq}, since we will be working in dimension $d=8$, for which the unperturbed problem remains unsolved.
The proof of Theorem~\ref{thm:NewSharpIneq} further requires the following key lemma, which amounts to an interesting sharp inequality on its own.
\begin{lemma}\label{lem:Reight}
    Let $g \in L^2(\R^8) \cap L^{\infty}(\R^8)$ be such that $\widehat{g}(\xi)=a-b|\xi|^2$ if $|\xi|\leq 4$, where $a\ge 16b$ and $b> b_\star:=\frac{2^{21}\pi^2}{5^2 7^2 11}$, and set $h={\bf 1}+g$. Then constants are the unique complex-valued maximizers of \eqref{20210524_17:54}. 
\end{lemma} 
\begin{proof}[Proof of Lemma~\ref{lem:Reight}]
We follow the same steps as in the proofs of Theorem~\ref{Thm1} and~\ref{Thm1b}; note that $\widehat{g}|_{\overline{B_4}}\ge 0$ due to our choice of $a$ and $b$, so the assumption (R1) is satisfied (in fact, the only purpose of the parameter $a$ is to ensure that (R1) holds). The kernel $K_h$, originally introduced in \S\ref{sec:cauchy_schwarz}, is now given by $K_{\bf 1}+a M_0-b M_2$. Here $K_{\bf 1}$ is given by \eqref{eq:kappad}. Following~\eqref{20210528_11:33}, we further consider
\begin{equation}\label{eq:K_zero}
    \widetilde{M}_0(\omega_1, \omega_2):=\int_{(\mathbb S^7)^2}\left( \lvert\omega_1+\omega_2\rvert^2 +\lvert\omega_3+\omega_4\rvert^2\right)\, \d\sigma(\omega_3)\d\sigma(\omega_4)= \sigma(\mathbb S^7)^2\, (4+2\omega_1\cdot\omega_2),
\end{equation}
so, setting $t=\omega_1\cdot \omega_2$, we have $M_0(t):= \widetilde M_0(\omega_1,\omega_2) = \sigma(\mathbb S^7)^2(4+2t)$. Finally,
 \[\widetilde M_2(\omega_1,\omega_2):=\int_{(\sph{7})^2}\left|\sum_{j=1}^4\omega_j\right|^2\left(|\omega_1+\omega_2|^2+|\omega_3+\omega_4|^2-(\omega_1+\omega_2)\cdot(\omega_3+\omega_4)\right) \d\sigma(\omega_3)\,\d\sigma(\omega_4)=\textup I+\textup{II}-\textup{III},
  \]
where each of the summands is defined as follows:
  \begin{align*}
  \textup I&:=|\omega_1+\omega_2|^2\int_{(\sph{7})^2}\left|\sum_{j=1}^4\omega_j\right|^2\d\sigma(\omega_3)\,\d\sigma(\omega_4),\\
 \textup {II}&:=\int_{(\sph{7})^2}|\omega_3+\omega_4|^2\left|\sum_{j=1}^4\omega_j\right|^2\d\sigma(\omega_3)\,\d\sigma(\omega_4),\\
  \textup {III}&:=\int_{(\sph{7})^2}(\omega_1+\omega_2)\cdot(\omega_3+\omega_4)\left|\sum_{j=1}^4\omega_j\right|^2\d\sigma(\omega_3)\,\d\sigma(\omega_4).
  \end{align*}
 Recall \eqref{20210720_23:08}, $\int_{(\sph{7})^2} |\omega_3+\omega_4|^2 \,\d\sigma(\omega_3)\,\d\sigma(\omega_4)=2\,\sigma(\sph{7})^2$. A simple change of variables $(\omega_3,\omega_4) \mapsto (-\omega_3,-\omega_4)$ yields
  \[\int_{(\sph{7})^2}(\omega_1+\omega_2)\cdot(\omega_3+\omega_4) \, |\omega_3+\omega_4|^k\,\d\sigma(\omega_3)\,\d\sigma(\omega_4)=0,\]
   for any $k\in\mathbb \Z_{\ge 0}$ and $(\omega_1,\omega_2)\in(\sph{7})^2$. From the identity 
$\int_{\sph{7}} (\omega\cdot \zeta)^2 \,\d\sigma(\omega) =  \frac{1}{8}\,  \sigma(\sph{7})$ for any $\zeta \in  \sph{7}$, we get
\begin{equation*}
    \int_{\mathbb S^7} \lvert \omega_3+\omega_4\rvert^4\, \d\sigma(\omega_3)= \tfrac92 \sigma(\mathbb S^7), 
\end{equation*}
and therefore
   \begin{equation*}
    \begin{array}{ccc} \frac{\textup{I}}{\sigma(\sph{7})^2}=2|\omega_1+\omega_2|^2+|\omega_1+\omega_2|^4,&
   \frac{\textup{II}}{\sigma(\sph{7})^2}= \frac92+2|\omega_1+\omega_2|^2, &
   \frac{\textup{III}}{\sigma(\sph{7})^2}=\frac12 |\omega_1+\omega_2|^2.
   \end{array}
  \end{equation*}
   Setting $t=\omega_1\cdot\omega_2$, we then have that
  \begin{equation}\label{eq:KgDef}
M_2(t):= \widetilde M_2(\omega_1,\omega_2)=\sigma(\sph{7})^2\left( \frac92 + \frac72(2+2t)+(2+2t)^2\right)=4\,\sigma(\mathbb S^7)^2\,t^2+P_1(t),
   \end{equation}
    where $P_1(t)$ is a degree 1 polynomial which will play no role in the following computations. Next we study the coefficients $\{\lambda_{8,h}(n)=\lambda_{8,{\bf 1}}(n)+a\lambda_{8,0}(n)-b\lambda_{8,2}(n)\}_{n\geq 0}$, analogous to the ones introduced in~\eqref{20210528_11:59}; here, 
\begin{equation}\label{eq:def_gegenbauer_coeff}\notag
    \begin{array}{ccc}
        \displaystyle \lambda_{8, \bf{1}}(n) = \sigma(\mathbb S^6)\int_{-1}^1\frac{C_n^3(t)}{C_n^3(1)}K_{\bf{1}}(t)(1-t^2)^\frac52\, \d t, & 
        \displaystyle \lambda_{8, k}(n) = \sigma(\mathbb S^6)\int_{-1}^1\frac{C_n^3(t)}{C_n^3(1)}M_k(t)(1-t^2)^\frac52\, \d t, & \text{for }k\in\{0, 2\}.
    \end{array}
\end{equation}

The proof will be complete once we show the analogous result to Lemma~\ref{Lem6_crux}; namely, that $\lambda_{8, h}(2\ell)<0$ for every $\ell\ge 1$.\footnote{Observe that $\lambda_{8, h}(0) > 0$.} Since $M_0(t)$ and $M_2(t)$ are polynomials of degree $1$ and $2$ respectively, $\lambda_{8, 0}(2\ell)=0$ for $\ell\ge 1$ and $\lambda_{8, 2}(2\ell)=0$ for $\ell\ge 2$.  Now, reasoning as in \cite[\S 5.5]{COS15}, we have that 
\begin{equation*}
    \begin{array}{cc}  
        \displaystyle\lambda_{8, \bf{1}}(2)
       =\frac{2^{20}\pi^{14}}{3^3 5^3 7^2 11}, & \text{ and }\lambda_{8, \bf{1}}(2\ell)<0\text{ for }\ell\ge 2.
    \end{array}
\end{equation*}
On the other hand, a straightforward computation reveals that $\lambda_{8, 2}(2)=\frac{\pi^{12}}{2 \times3^3\times 5}$; noticing that $b_\star=\frac{\lambda_{8, \bf{1}}(2)}{\lambda_{8, 2}(2)}$, we then have $\lambda_{8, h}(2)=\lambda_{8, \bf{1}}(2)-b\lambda_{8, 2}(2) < 0$, and we conclude that constant functions maximize~\eqref{20210524_17:54}. The uniqueness statement follows from the argument of Section \ref{sec:classification}, as in the proof of Proposition~\ref{prop:tildeLd}. 
\end{proof}
\begin{proof}[Proof of Theorem~\ref{thm:NewSharpIneq}] Let $\widehat{g}_{a, b}(\xi)=(a-b\lvert\xi\rvert^2){\bf 1}_{\overline{B_4}}(\xi)$ and $\widehat{h}_b(\xi)=b\lvert\xi\rvert^2{\bf 1}_{\overline{B_4}}(\xi)$. By Lemma \ref{lem:Reight}, the first of the following two ratios is maximized when $f$ is a constant function on $\mathbb S^7$:
\begin{equation}\label{eq:weighted_ratios}
    \begin{array}{cc}
        \displaystyle \frac{ \displaystyle \int_{\mathbb R^8} \lvert \widehat{f\sigma}(x)\rvert^4 (1+g_{a, b}(x))\, \d x}{\lVert f \rVert_{L^2(\mathbb S^7)}^4}, &        \displaystyle \frac{ \displaystyle\int_{\mathbb R^8} \lvert \widehat{f\sigma}(x)\rvert^4 h_{ b}(x)\, \d x}{\lVert f \rVert_{L^2(\mathbb S^7)}^4}.
    \end{array}
\end{equation}
The proof of Proposition \ref{prop:tildeLd} reveals that the second ratio is also maximized when $f$ is constant. 
Indeed, by computations analogous to~\eqref{eq:mod_cauchy_schwarz} and~\eqref{eq:mod_kernel}, we have that
\begin{equation}\label{eq:modmod_cauchy_schwarz}
    \begin{split}
        \int_{\mathbb R^8} (\widehat{f\sigma}(x))^4h_b(x)\, \d x & \le 2\,b\,\sigma(\mathbb S^7)^2\int_{(\mathbb S^7)^2}f(\omega_1)^2f(\omega_2)^2(2+\omega_1\cdot\omega_2)\, \d \sigma(\omega_1)\d \sigma(\omega_2) \\ 
        &=4\,b\,\sigma(\mathbb S^7)^2\lVert f\rVert_{L^2(\mathbb S^7)}^4,
    \end{split}
\end{equation}
with equality for constant $f$, provided that $f\in L^2(\mathbb S^7)$ is non-negative and even; recall that these extra assumptions can be incorporated by Lemma~\ref{Lem2_non-negative} and Lemma~\ref{Lem3_even}, respectively. 

\smallskip

Therefore, the ratio $\lVert f\rVert_{L^2(\mathbb S^7)}^{-4}\int_{\mathbb R^8} \lvert\widehat{f\sigma}(x)\rvert^4(1+g_{a, b}(x)+h_b(x))\, \d x$ is also maximized by constant $f$. To conclude, observe from \eqref{eq:Q_h_decomp} that 
\begin{equation*}
    \int_{\mathbb R^8} \lvert\widehat{f\sigma}(x)\rvert^4(1+g_{a, b}(x)+h_b(x))\, \d x  =  \int_{\mathbb R^8} \lvert \widehat{f\sigma}(x)\rvert^4\, \d x +a\left| \int_{\mathbb S^7}  f(\omega)\, \d \sigma(\omega)\right|^4,
\end{equation*}
which is the left-hand side of the sought inequality~\eqref{eq:corrector_estimate}. This completes the proof.
\end{proof}

\section*{Acknowledgements}
EC acknowledges support from FAPERJ - Brazil. GN\@ and DOS\@ are supported by the EPSRC New Investigator Award ``Sharp Fourier Restriction Theory'', grant no.\@ EP/T001364/1.
DOS\@ acknowledges further partial support from the DFG under Germany's Excellence Strategy -- EXC-2047/1 -- 390685813. 
The authors are grateful to Dimitar Dimitrov, Felipe Gonçalves, Jo\~ao Pedro Ramos  and Luis Vega for insightful discussions during the preparation of this work. {The authors are also grateful to the anonymous referees for  valuable suggestions}.

\end{document}